\documentclass[11pt]{amsart}
\usepackage[margin=1in]{geometry}
\usepackage{amsfonts}
\usepackage{amssymb}
\usepackage[dvips]{graphics}
\usepackage{epsfig}
\usepackage{euscript}
\usepackage{color}

 \newtheorem{thm}{Theorem}[section]
 \newtheorem{cor}[thm]{Corollary}
 \newtheorem{lem}[thm]{Lemma}
 \newtheorem{prop}[thm]{Proposition}
 \theoremstyle{definition}
 
 \theoremstyle{remark}
 \newtheorem{rem}[thm]{Remark}
 
 \numberwithin{equation}{section}

 \def\idty{{\mathchoice {\mathrm{1\mskip-4mu l}} {\mathrm{1\mskip-4mu l}} %
{\mathrm{1\mskip-4.5mu l}} {\mathrm{1\mskip-5mu l}}}}

\newcommand{\R}{{\mathbb R}}
\newcommand{\bC}{{\mathbb C}}

\newcommand{\Z}{{\mathbb Z}}

\newcommand{\supp}{\operatorname{supp}}

\newcommand{\om}{\omega}
\newcommand{\la}{\lambda}

\newcommand{\eps}{\varepsilon}
\newcommand{\norm}[1]{\left\Vert#1\right\Vert}
\newcommand{\abs}[1]{\left\vert#1\right\vert}
\newcommand{\dist}{{\ensuremath{\mathrm{dist}}}}
\newcommand{\sign}{\mathrm{sign}}
\newcommand{\diam}{\mathrm{diam}}

\newcommand{\be}{\begin{equation}}
\newcommand{\ee}{\end{equation}}

\begin{document}
\title{The Klein-Gordon Equation on $\Z^2$ and the Quantum Harmonic Lattice}
\author{Vita Borovyk}
\address{Department of Mathematics, University of Cincinnati,
Cincinnati, OH 45221-0025}
\email{Vita.Borovyk@uc.edu}
\author{Michael Goldberg}
\thanks{The second author received support from NSF grant DMS-1002515 during the 
preparation of this work.}
\address{Department of Mathematics, University of Cincinnati,
Cincinnati, OH 45221-0025}
\email{Michael.Goldberg@uc.edu}
\thanks{Both authors would like to thank Jonathan Goodman for sharing his expertise
on several occasions.}

\begin{abstract}
The discrete Klein-Gordon equation on a two-dimensional square lattice satisfies
an $\ell^1 \mapsto \ell^\infty$ dispersive bound with polynomial decay rate
$|t|^{-3/4}$.  We determine the shape of the light cone for any choice of the
mass parameter and relative propagation speeds along the two coordinate axes.
Fundamental solutions experience the least dispersion along four caustic lines
interior to the light cone rather than along its boundary, and decay
exponentially to arbitrary order outside the light cone.  The overall geometry of
the propagation pattern and its associated dispersive bounds are independent of
the particular choice of parameters.  In particular there is no bifurcation of the
number or type of caustics that are present.

The discrete Klein-Gordon equation is a classical analogue of the quantum
harmonic lattice.  In the quantum setting, commutators of time-shifted
observables experience the same decay rates as the corresponding
Klein-Gordon solutions, which depend in turn on the relative location of the
observables' support sets.

\end{abstract}

\maketitle

\section{Introduction}

The wave equation $u_{tt} - \Delta u = 0$ on $\R^{2+1}$ is explicitly solved via
Poisson's formula, in which initial data $u(x,0) = g(x)$, $u_t(x,0) = h(x)$
determines the unique solution
\begin{equation*}
u(x,t) = \frac{\sign(t)}{2\pi}\int_{|y-x| < |t|} \frac{h(y) + 
\frac{1}{t}(g(y) + \nabla g(y)\cdot(y-x))}{\sqrt{t^2-|y-x|^2}}\,dy
\end{equation*}
at any $t \not= 0$.  More generally the Klein-Gordon equation
$u_{tt} - \Delta u + m^2 u = 0$ with the same initial data has the solution
\begin{equation*}
u(x,t) = \frac{\sign(t)}{2\pi}\int_{|y-x| < |t|} \Big(h(y) + 
\frac{1}{t}(g(y) + \nabla g(y)\cdot(y-x))\Big)
\frac{\cos(m\sqrt{t^2 - |y-x|^2}\,)}{\sqrt{t^2-|y-x|^2}}\,dy.
\end{equation*}
When $m = 0$ the two equations coincide.  It is clear in these formulas 
that the propagator kernel is radially symmetric, and that all information
from the initial data travels at finite speed.  Both equations also
possess a well-known dispersive property that $|u(x,t)| \leq C|t|^{-1/2}$
provided the initial data are sufficiently smooth and decay at infinity.

This paper investigates the propagation patterns and dispersive bounds
for solutions to a family of discrete Klein-Gordon equations on
$\Z^2\times\R^1$ with distinct coupling strength along the two coordinate directions
as suggested below.
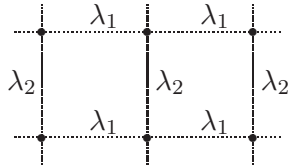
\begin{figure}[ht]
\begin{picture}(120,75)
\thinlines
\qbezier[60](15,65)(65,65)(115,65)
\qbezier[60](15,25)(65,25)(115,25)
\qbezier[60](25,75)(25,45)(25,15)
\qbezier[60](65,15)(65,45)(65,75)
\qbezier[60](105,15)(105,45)(105,75)
\put(43,28){$\lambda_1$}
\put(12,43){$\lambda_2$}
\put(43,68){$\lambda_1$}
\put(68,43){$\lambda_2$}
\put(85,28){$\lambda_1$}
\put(85,68){$\lambda_1$}
\put(108,43){$\lambda_2$}

\put(25,25){\circle*{3}}
\put(25,65){\circle*{3}}
\put(65,25){\circle*{3}}
\put(65,65){\circle*{3}}
\put(105,25){\circle*{3}}
\put(105,65){\circle*{3}}
\end{picture}
\label{lattice}
\end{figure}

The discrete Klein-Gordon equation for this system with Cauchy initial data is 
\begin{equation} \label{eq:discKleinGordon}
\left\{
\begin{aligned}
&u_{tt}(x,t) \underbrace{- {\textstyle \sum_{j=1}^2}\lambda_j 
\big(u(x+e_j,t) + u(x-e_j,t) - 2u(x,t)\big)
+ \omega^2 u(x,t)}_{Hu} = 0 \\
&u(x,0) = g(x) \\
&u_t(x,0) = h(x)
\end{aligned} \right.
\end{equation}
with $\omega, \lambda_1, \lambda_2 >0$ fixed parameters.  It has a conserved
energy functional
\begin{equation}
E(t) = \frac{1}{2}\sum_{x \in \Z^2} \Big(u_t^2(x,t) + \omega^2 u^2(x,t)
+ \sum_{j=1}^2 \lambda_j(u(x+e_j,t) - u(x,t))^2\big)
\end{equation}
and (by analogy with the continuous setting) the values of $\lambda_j$
suggest propagation speeds of $\sqrt{\lambda_j}$ along their respective
coordinate directions.  We have chosen the letter $\omega$ for the mass
parameter in order to highlight connections between~\eqref{eq:discKleinGordon}
and the quantum harmonic lattice system.  Those connections are explored in
further detail in section~\ref{sec:quantum}.

The one-dimensional discrete wave equation provides a certain degree of inspiration;
when $\lambda = 1$
its fundamental solution is expressed in terms of the Bessel functions
$J_{|y-x|}(t)$ (see \cite{marchioro1978}). 
There are three main asymptotic regimes.  For $|t| \gg |y-x|$
there is oscillation with amplitude $|t|^{-1/2}$. When $|t| \ll |y-x|$ the
propagator is nonzero (hence there is some rapid transfer of information)
with exponential decay at spatial infinity on the order of 
$((2|y-x|)^{-1}et)^{|y-x|}$.
For $|y-x| = |t| + O(t^{1/3})$ the propagator kernel reaches its maximum size
of approximately $|t|^{-1/3}$.  This bound is most easily obtained by applying 
van der Corput's lemma to the Fourier representation of $J_{|y-x|}(t)$ .

Analysis of the discrete Schr\"odinger equation on $\Z$ yields a similar structure
for its fundamental solution.  Moreover, the Schr\"odinger equation on $\Z^d$
separates into a product of one-dimensional solutions, and therefore attains a maximum
value comparable to $|t|^{-d/3}$ near the corners of a box with side lengths
$\lambda_j |t|$.

Unfortunately the discrete wave and Klein-Gordon equations in higher dimensions 
do not separate variables in the same way.  The fundamental
solution can still be determined as a superposition of plane waves, with size
bounds in the different regimes resulting from stationary phase principles.
Isotropic wave equations ($\lambda_j = 1$, $\omega = 0$) in two and three
dimensions were analyzed by Schultz~\cite{schultz98},
with a curious
set of outcomes.  In addition to the expected wavefront expanding radially at
$|y-x| \sim |t|$, there is a secondary region of reduced dispersion traveling at
somewhat lower speed.  In two dimensions the region lies along an 
astroid-shaped curve with diameter $\sqrt{2}\,|t|$; in three dimensions
the region follows a cusped and pointed surface of a similar nature.
Surprisingly, some global dispersive bounds are dominated by behavior when
$|y-x|$ belongs to the secondary set, even though this occurs well inside
the overall propagation pattern.

In the two-dimensional discrete Klein-Gordon equation, 
each plane wave $u_k(x) := e^{ik\cdot x}$ 
satisfies
$Hu_k = \gamma^2(k) u_k$, with the dispersion relation
\[
\gamma^2(k) = \omega^2 + \sum_j 2\lambda_j(1-\cos k_j)
\]
and $k$ ranging over the fundamental domain
$[-\pi, \pi]^2$.  The solution of~\eqref{eq:discKleinGordon}
is given formally by
\begin{equation*}
u(x,t) = \cos(t\sqrt{H})g + \frac{\sin(t\sqrt{H})}{\sqrt{H}} h
\qquad {\rm and} \qquad
u_t(x,t) = - \sqrt{H}\,\sin(t\sqrt{H}) g + \cos(t\sqrt{H}) h.
\end{equation*}
The operators involved act on a plane wave $u_k$ by multiplication by
$\cos(t\gamma(k))$,
$\frac{\sin(t\gamma(k))}{\gamma(k)}$, and $-\gamma(k)\sin(t\gamma(k))$,
hence the fundamental solutions of~\eqref{eq:discKleinGordon} in physical
space will be the inverse Fourier transform of those three functions.
For all practical purposes these are oscillatory integrals over the torus
$k \in [-\pi,\pi]^2$ of the form given in~\eqref{eq:intx},
 whose asymptotic behavior is governed by critical points
of the phase function $t\gamma(k) \pm x\cdot k$.

Three distinct regimes again emerge:
critical points are absent for $ |x| \gg |t|$ and exponential decay
is observed by following the analytic continuation of $\gamma(k)$ into
$\{[-\pi,\pi] + i\R\}^2$.  Quantitative exponential bounds are given in
Theorems~\ref{thm:exp1} and~\ref{thm:exp2}.  
For generic values of $(x,t)$ inside the
``light cone'' (i.e. $x = t\nabla \gamma(k)$ for some $k$),
stationary phase arguments lead to a bound of $|t|^{-1}$.
Along the boundary of the light cone, and within the secondary region introduced
above, degenerate stationary phase estimates yield polynomial time decay with
a fractionally smaller exponent.  For fixed $t \not= 0$ there is a global bound
of order $|t|^{-3/4}$ with maxima occurring near the four
cusps of the astroid curve.  This is a faster rate of decay than the discrete
Schr\"odinger equation on $\Z^2$, where separation of variables leads to a
$|t|^{-2/3}$ bound instead.  Further details about the structure of the Klein-Gordon
propagators are summarized in Theorem~\ref{thm:poly} and its corollaries.

Combining the pointwise $|t|^{-3/4}$ bound with exponential decay outside
of the light cone yields a family of estimates for the Klein-Gordon propagator
as a map from $\ell^1(\Z^2)$ to $\ell^p(\Z^2)$.  These are useful for establishing
global existence of small solutions to the nonlinear discrete Klein-Gordon equation
as suggested by~\cite{mielkepatz}.  For power-law nonlinearities
\begin{equation} \label{eq:discNLKG}
\left\{
\begin{aligned}
&u_{tt}(x,t) - {\textstyle \sum_{j=1}^2}\lambda_j 
\big(u(x+e_j,t) + u(x-e_j,t) - 2u(x,t)\big)
+ \omega^2 u(x,t) = |u|^{\beta - 1}u \\
&u(x,0) = g(x) \\
&u_t(x,0) = h(x)
\end{aligned} \right.
\end{equation}
our results imply that small initial data produce global solutions when $\beta > 4$.
We introduce this application immediately after stating all the relevant bounds,
at the end of Section~\ref{sec:results}.

The endpoint case $\omega = 0$ corresponds to the discrete wave equation, and it
introduces a new type of asymptotic behavior at the boundary of the light cone
because the phase function $\gamma(k)$ develops an absolute-value singularity at
the origin.  Schultz computed the resulting
dispersive estimate in~\cite{schultz98} under the further assumption
$\lambda_1 = \lambda_2$.  We provide a restatement of these bounds
for general $\lambda_j$ in Section~\ref{sec:wave}.
Detailed calculations in a neighborhood of the light-cone boundary 
are identical to the ones in~\cite{schultz98} and we do not repeat them here.

Technical notes:
The problem of generalizing van der Corput's lemma to two and higher
dimensions has a long history in harmonic analysis.  It is roughly equivalent
to determining the area of level sets or constructing a resolution of
singularities for smooth functions.  Schultz~\cite{schultz98} computed dispersive
bounds for the isotropic wave equation by exhibiting an explicit unfolding for
each of the fold and cusp singularities that arise.  The same methods could be
employed here, though the dependence on the coupling parameters $\lambda_j$
is unduly complicated. We instead follow Varchenko's 1976 exposition~\cite{Var}
which permits estimation of the oscillatory integral directly from the Taylor series
of the phase function provided the coordinate system is sufficiently well
"adapted."
Modern techniques for the general resolution of singularities may be
needed for applications where the domain has a more intricate periodic
structure than $\Z^2$, and especially in high-dimensional settings.
In those cases one may employ methods and results by
Greenblatt~\cite{greenblatt} in two dimensions or
Collins, Greenleaf, and Pramanik~\cite{CoGrPr2010} in higher dimensions.
At one point in Corollary~\ref{cr:OmegakCk} we also invoke a recent 
result by Ikromov and M\"uller~\cite{IkMu}
regarding the stability of degenerate integrals under linear
perturbation of the phase.

The fact that only fold and cusp singularities appear in this problem
is noteworthy in itself.  
Unlike in the continuous setting, varying parameters $\lambda_1$, $\lambda_2$
is not equivalent to performing a diagonal linear transformation on $x$ because
the domain $\Z^2$ and its Fourier dual both lack a dilation symmetry.  
We show in this paper that
the dispersion pattern for~\eqref{eq:discKleinGordon} retains the same
topological and geometric structure found in~\cite{schultz98} for all values of
$\omega, \lambda_j > 0$.  In particular there is no choice of parameters
that generates exceptional degeneracy or bifurcation of the phase function
singularities which determine the dispersive estimate.
Separately we show that interactions outside the
light cone are all subject to an exponential bound, and that
exponential bounds of any desired order are achieved by setting $|y-x|/|t|$
sufficiently large.  The latter statement is akin to (and readily implies)
Lieb-Robinson bounds (cf. \cite{lieb1972}, \cite{marchioro1978}, \cite{butta2007},
\cite{eisert2008}, \cite{raz2009}, \cite{nachtergaele2009}, \cite{NS3}) for the
corresponding quantum harmonic lattice.  The background and details of this
application are presented in Section~\ref{sec:quantum}.


The paper is organized as follows: Section~\ref{sec:results} enumerates the precise statements and 
illustrations of our main result, describes an application to small-data global existence for
some nonlinear discrete Klein-Gordon equations, 
and remarks about linear behavior in the wave-equation ($\omega = 0$) limit.
A proof of the main theorems is sketched out in the
following section, assuming a number of propositions about the critical
points of $t\gamma(k) - x\cdot k$.  Section~\ref{sec:quantum} translates our results
about the discrete Klein-Gordon equation into dynamical properties of the
quantum harmonic lattice. All of the relevant assertions about properties of
$\gamma(k)$ are finally proved in the conluding section, using more or less bare-handed
calculation of its Taylor series expansions.

\section{Main results}  \label{sec:results}
In this section we state the result describing the long-time behavior of the solution of \eqref{eq:discKleinGordon}. We start with estimates of a slightly more general oscillatory integral and our main result is based on these estimates.

Choose a set of values $\omega, \lambda_1, \lambda_2 > 0$.  For the function
\begin{equation}																											\label{gammaABk}
\gamma(k) = \bigg(\omega^2 + \sum_{j=1}^2 2\lambda_j(1-\cos k_j)\bigg)^{1/2}, \qquad k=(k_1, k_2) \in [-\pi, \pi]^2,
\end{equation}
introduce
\begin{equation}																											\label{eq:intx}
I(t, x, \eta) = \frac 1 {(2\pi)^2} \int_{[-\pi, \pi]^2} e^{i(k \cdot x - t\gamma(k))} \eta(k)\, dk,
\end{equation}
where $x \in \Z^2$, $t \in \R$, and $\eta$ is a smooth test function on $[-\pi, \pi]^2$ with periodic boundary conditions.
The asymptotic behavior of oscillatory integrals is generally influenced by
local considerations, in which case $\eta(x)$ may be assumed to have compact
support in a fundamental domain of $\R^2/2\pi\Z^2$
(for example as part of a partition of unity).
In that case $\eta$ can be extended by zero to a function on all of
$\R^2$ and one may define for all $x \in \R^2$
\begin{equation}																											\label{eq:Itilde}
\tilde{I}(t, x, \eta) = \frac 1 {(2\pi)^2} \int_{\R^2} e^{i(k \cdot x - t\gamma(k))} \eta(k)\, dk,
\end{equation}
up to a unimodular constant whose value is exactly 1 if $x \in \Z^2$.
It is often convenient to consider $x$ of the form $x = v t$, with $v$ a fixed vector in
$\R^2$ (representing velocity), so that the integral \eqref{eq:Itilde} can be written as
\begin{equation}											\label{eq:intv}
\tilde{I}(t, vt, \eta) = \frac 1 {(2\pi)^2} \int_{\R^2} e^{it\phi_v(k)} \eta(k)\, dk,
\ \ {\rm where}\ \ \phi_v(k) := k\cdot v - \gamma(k).
\end{equation} 
Estimates on $I(t, x, \eta)$ will follow from restricting the corresponding
bound on~\eqref{eq:intv} to examples with $x = vt \in \Z^2$.

The long-time behavior of \eqref{eq:intv} is dictated by the highest level of degeneracy of the phase within the support of $\eta$. In the absence of critical points, integration by parts multiple times yields a rapid (faster than polynomial) decay. In the case where critical points are present but are non-degenerate, the standard stationary phase argument provides $|t|^{-1}$ decay for the integral. Finally, if there are degenerate critical points, the decay is slower and more careful analysis is needed to determine its exact order. It is easy to see that for any point $k^* \in [-\pi, \pi]^2$, there is a choice of the velocity $v$ such that $k^*$ is a critical point of $\phi_v(k)$ (namely, $v = \nabla \gamma(k^*)$). The order of degeneracy of the phase at that point is determined by second and higher-order derivatives of $\gamma$, as the linear component is canceled by subtracting $k^* \cdot v$.  We introduce a partition of $[-\pi, \pi]^2$ with respect to the degeneracy order of $\gamma$,
\begin{equation}
[-\pi, \pi]^2 = K_1 \cup K_2 \cup K_3,
\end{equation}
where
\begin{align}
\notag K_1 &= \{k \in [-\pi, \pi]^2: \det D^2 \gamma(k) \ne 0\},\\
\label{eq:setK} K_2 &= \{k \in [-\pi, \pi]^2: \det D^2 \gamma(k) = 0, \; (\xi \cdot \nabla)^3\gamma(k) \ne 0\},\\
\notag K_3 &= \{k \in [-\pi, \pi]^2: \det D^2 \gamma(k) = 0, \; (\xi \cdot \nabla)^3\gamma(k) = 0\}.
\end{align}
In the definition of $K_2$ and $K_3$, $\xi$ stands for an eigenvector of the $2\times 2$ matrix $D^2 \gamma(k)$ corresponding to the zero eigenvalue. 

The analysis of Section~\ref{sec:phasefunc} allows us to describe the structure of this partition in detail.
Proposition~\ref{prop:rank1} notes in particular that the rank of $D^2\gamma(k)$ is never zero, so the
direction of $\xi$ is always well defined.
\begin{lem} \label{lem:Kpicture}
For every choice of $\omega, \lambda_1, \lambda_2 > 0$, 
the sets $K_i$, $i = 1, 2, 3$, defined in \eqref{eq:setK} possess the following properties.
\begin{itemize}
\item $K_3$ consists of four points related by mirror symmetry across the coordinate axes.
\item $K_2$ consists of two closed curves, one around the origin and the other around the point $(\pi, \pi)$,
with the four points of $K_3$ removed from the latter curve.
\item $K_1 = [-\pi, \pi]^2 \setminus (K_3 \cup K_2)$. This set consists of three open regions: the interior of the small closed curve around zero, the interior of the closed curve around the point $(\pi, \pi)$, and the area of the compactified torus enclosed between these two curves. 
\end{itemize}
The structure of the partition is displayed in Figure~\ref{K123} (similar to Figure~\ref{PhiKfig}).
\begin{figure}[ht]
\begin{picture}(160,160)
\put(0,70){\vector(1,0){160}}											
\put(165,70){$k_1$}
\put(80,0){\vector(0,1){150}}
\put(65,147){$k_2$}
\thinlines
\qbezier(30,120)(30,120)(130,120)									
\qbezier(30,20)(30,20)(130,20)
\qbezier(30,120)(30,120)(30,20)
\qbezier(130,120)(130,120)(130,20)
\put(73,124){$\pi$}
\put(134,75){$\pi$}
\put(14,75){$-\pi$}
\put(65,10){$-\pi$}
\thicklines
\qbezier(30,55)(65,55)(65,20)											
\qbezier(30,85)(65,85)(65,120)
\qbezier(95,20)(95,55)(130,55)
\qbezier(95,120)(95,85)(130,85)
\qbezier(70,70)(70,60)(80,60)											
\qbezier(90,70)(90,60)(80,60)	
\qbezier(90,70)(90,80)(80,80)	
\qbezier(80,80)(70,80)(70,70)	
\put(112,88){\circle*{4}}														
\put(112,51){\circle*{4}}
\put(48,88){\circle*{4}}
\put(48,51){\circle*{4}}
\end{picture}
\caption{Sets $K_1$, $K_2$, $K_3$}
\label{K123}
\end{figure}
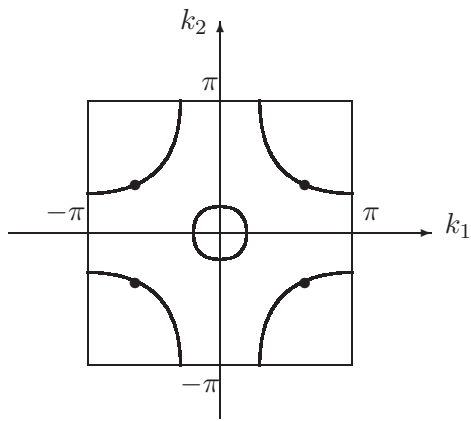 
\end{lem}
The statements in Lemma~\ref{lem:Kpicture} follow immediately from
equation~\eqref{eq:Phi1} and Corollary~\ref{cor:Kstar}.

\begin{rem}
The definition of $K_3$ provides a minimum degree of degeneracy for $\phi_v(k)$
at each point $k^* \in K_3$ (with $v = \nabla \gamma(k^*)$), but it does not
specify an upper bound.  Naive dimensional analysis suggests that there may
exist exceptional values of $\omega, \lambda_j$ for which the next higher order
derivative of $\gamma$ also vanishes on $K_3$.  We show in
Lemma~\ref{lm:fourthder} that in fact the opposite is true.  In other words, the
singularities of $\phi_v$ are stable globally within the parameter space
$\{\omega, \lambda_1, \lambda_2 > 0\}$.
\end{rem}

To see the connection between velocities $v \in \R^2$ and possible degeneracies of $\phi_v$, consider the images of sets $K_i$ in the velocity space:
\begin{align}
\notag V_3 &= \{v \in \R^2: \text{ there exists } k \in K_3 \text{ such that } v = \nabla \gamma (k)\},\\
\label{eq:setsV} V_2 &= \{v \in \R^2\setminus V_3: \text{ there exists } k \in K_2 \text{ such that } v = \nabla \gamma (k)\},\\
\notag V_1 &= \{v \in \R^2\setminus (V_2 \cup V_3): \text{ there exists } k \in K_1 \text{ such that } v = \nabla \gamma (k)\},\\
\notag V_0 &= \{v \in \R^2: \text{ for all } k \in [-\pi, \pi]^2,\; v \ne \nabla \gamma (k)\}.
\end{align}
Alternatively, under the mapping $\mathcal{V}:[-\pi, \pi]^2 \to \R^2$ defined by 
\begin{equation}																			\label{eq:nablagammamap}
\mathcal{V}(k) = \nabla \gamma (k),
\end{equation}
sets \eqref{eq:setsV} admit the representation
\begin{align}
\notag              V_3 &= \mathcal{V}(K_3),\\
\label{eq:setsVmap} V_2 &= \mathcal{V}(K_2),\\
\notag              V_1 &= \mathcal{V}(K_1)\setminus (V_3 \cup V_2),\\
\notag              V_0 &= \R^2\setminus \mathcal{V}([-\pi, \pi]^2).
\end{align}
\begin{prop} \label{prop:Vpicture}
Fix $\omega, \lambda_1, \lambda_2 > 0.$
Let the sets $\{V_i\}_{i=0}^3$ be defined by \eqref{eq:setsV}. Then they are located as shown on Figure~\ref{Velocities}. There are two simple closed continuous curves $\Psi_1$ and $\Psi_2$ around the origin that split the plane into three open regions. More precisely, $\Psi_1$ encloses a convex region and $\Psi_2$ consists of four concave arcs that meet at cusps.  The four vertices of these cusps form the set $V_3$. The union of $\Psi_1$ and $\Psi_2$, with the four points that belong to $V_3$ removed, is $V_2$.  The union of the two inner bounded open regions  is $V_1$. The unbounded region is $V_0$ which has boundary curve $\Psi_1$. 
\begin{figure}[ht]
\begin{picture}(200,200)
\put(0,90){\vector(1,0){200}}											
\put(203,90){$x_1$}
\put(100,15){\vector(0,1){155}}
\put(95,173){$x_2$}
\qbezier(70,122)(80,90)(70,58)											
\qbezier(70,122)(100,110)(130,122)
\qbezier(130,122)(120,90)(130,58)
\qbezier(130,58)(100,70)(70,58)
\qbezier(35,90)(33,32)(100,35)											
\qbezier(165,90)(160,30)(100,35)	
\qbezier(165,90)(170,150)(100,150)	
\qbezier(100,150)(30,150)(35,90)	
\put(130,122){\circle*{4}}                         
\put(130,58){\circle*{4}}
\put(70,122){\circle*{4}}
\put(70,58){\circle*{4}}
\put(10,140){$V_0$}															
\put(177,130){$V_1$}
\put(175,130){\vector(-1,-1){30}}	
\put(175,130){\vector(-2,-1){65}}	
\put(50,70){$V_2$}
\put(48,72){\vector(-1,0){10}}	
\put(62,72){\vector(1,0){10}}
\put(132,58){$V_3$}
\put(132,122){$V_3$}
\put(70,125){$V_3$}
\put(72,48){$V_3$}
\put(49,95){$\Psi_1$}
\put(49,105){$\Psi_2$}
\put(48,97){\vector(-1,0){10}}	
\put(62,107){\vector(1,0){10}}
\end{picture}
\caption{Sets $V_0$, $V_1$, $V_2$, and $V_3$.}
\label{Velocities}
\end{figure}
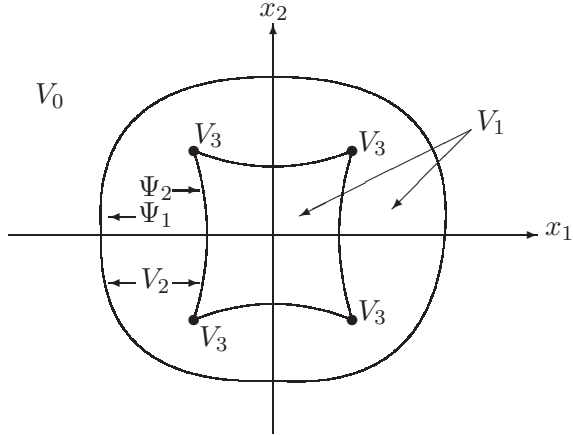 
\end{prop}

We can now state the main result relating a velocity $v$ to the decay order of an oscillatory integral \eqref{eq:intv} with phase function $\gamma$. 

\begin{thm}																													\label{thm:poly}
Fix $\omega, \lambda_1, \lambda_2 > 0$.
Let the sets $\{V_i\}_{i=1}^3$ be defined by \eqref{eq:setsV}, $\eta$ be a smooth periodic function on $[-\pi, \pi]^2$, and integral $I(t, x, \eta)$ defined by \eqref{eq:intx}. Then for any fixed $\delta > 0$, there exist constants $C_0$,$C_1$, $C_2$ and $C_3$ depending on $\eta$ such that 
\begin{align}
\label{eq:poly34} \text{ for all } & x \text{ with } \dist\left(x, t V_3\right) \le t\delta, \text{ we have } |I(t, x, \eta)| \le \frac {C_3}{|t|^{3/4}},\\
\label{eq:poly56} \text{ for all } & x \text{ with } \dist\left(x, t V_3\right) > t\delta \text{ and } \dist\left(x, t V_2\right) \le t\delta, \text{ we have } |I(t, x, \eta)| \le \frac {C_2}{|t|^{5/6}},\\
\label{eq:poly1} \text{ for all } & x \text{ with } \dist\left(x, t (V_3 \cup V_2)\right) > t\delta \text{ and } \dist(x, tV_1) \leq t\delta, \text{ we have } |I(t, x, \eta)| \le \frac {C_1}{|t|}, \\
\label{eq:polyN} \text{given } N &\geq 1, \text{ for all } x \text{ with } \dist(x, t(V_3 \cup V_2 \cup V_1)) > \delta,  \text{ we have } |I(t,x,\eta)| \le \frac{C_0}{|t|^N}.
\end{align} 
Each of $C_0$, $C_1$, and  $C_2$ depend on $\delta$ (and $C_0$ also depends on $N$). The values of $C_0$ -- $C_2$ are expected to approach infinity as $\delta$ approaches zero, while $C_3$ is independent of $\delta$. 
\end{thm}

If one is only concerned with the worst possible decay among all $x \in \Z^2$, the simpler statement is as follows.
\begin{cor}								\label{thm:unif}
Fix $\omega, \lambda_1, \lambda_2 > 0$. Let the integral $I(t, x, \eta)$ be defined by \eqref{eq:intx}. Then there exists $C = C(\eta) > 0$ such that for all $x \in \Z^2$,
\begin{equation} \label{eq:global34}
\displaystyle |I(t, x, \eta)| \le \frac {C}{|t|^{3/4}}.
\end{equation}
\end{cor}

\begin{cor} \label{cor:poly}
As a special case of Theorem~\ref{thm:poly}, the propagators of the discrete Klein-Gordon equation~\eqref{eq:discKleinGordon}
are recovered by choosing $\eta = \gamma^m$, $m = -1, 0, 1$.  Hence the values of $u(x,t)$
and $u_t(x,t)$ for a solution of~\eqref{eq:discKleinGordon} both satisfy~\eqref{eq:poly34}--\eqref{eq:global34}, 
provided the initial data $g$, $h$ are supported at the origin.  The result extends to
all $g$, $h$ supported in $B(0,R)$ by superposition, once enough time has elapsed
that  $\delta t > 2R$.
\end{cor}

The values of the exponents in~\eqref{eq:poly34}--\eqref{eq:poly1} are dictated by the worst degeneracy degree of critical points of the phase function. For example, if a velocity belongs to the set $V_1$ and is relatively far from $V_2$ and $V_3$, all the critical points of $\phi_v$ will be uniformly non-degenerate. In this case the decay rate of an oscillatory integral is $|t|^{-d/2}$ in arbitrary dimension $d$ (hence it is $|t|^{-1}$ in dimension two). 

Let us now briefly describe how the rates produced by velocities that are near $V_2$ and $V_3$ are computed (see Section \ref{sec:polyproof} for details). 
If $v \in V_2 \cup V_3$, then $\phi_v$ has at least one  degenerate critical point
 $k^* \in [-\pi,\pi]^2$.  Then the Taylor series expansion of $\phi_v$ near its critical point takes the form

\begin{equation}			\label{eq:phasephiv}
\phi_v(k) = k \cdot v - \gamma(k) = c_0 + \sum_{\substack{n,m\ge 0 \\ n+m \geq 2}} c_{n,m}(k_1-k^*_1)^n (k_2-k^*_2)^m.
\end{equation}

This Taylor series is said to be supported on the set of indices
$(m,n)$ where $c_{m,n} \not=0$.  Roughly speaking, one computes the
leading-order decay of the corresponding oscillatory integral by measuring
the distance from the origin to the convex hull of the Taylor series support,
then taking the reciprocal.  However the support is not invariant under
changes of coordinates, so one must first choose an ``adapted" coordinate
system that maximizes this distance~\cite{Var}.  
It turns out that for each function $\phi_v(k)$ the linear coordinate system that diagonalizes the Hessian matrix is adapted (Lemma~\ref{lm:fourthder} verifies this property in the one case where it is not readily apparent).
According to the definitions \eqref{eq:setK} and \eqref{eq:setsV}, the Taylor series of $\phi_v$ with respect to these coordinates has more vanishing low-order terms  if $v \in V_3$  as compared to $v \in V_2$.  Therefore its Newton polyhedron lies further away from the origin, and the oscillatory integral decays more slowly.  The Newton polyhedra associated with $v \in V_2$ and $v \in V_3$ are sketched in Figures~\ref{pic:NP2} and~\ref{pic:NP3} and give rise to the exponents in~\eqref{eq:poly56} and~\eqref{eq:poly34} respectively.

If $0 < \dist(v, V_2 \cup V_3) < \delta$, then $\phi_v$ does not have a
degenerate critical point itself, but it is related to the degenerate phase
functions described above by a small linear perturbation.  The fact that
oscillatory integral estimates are stable under such perturbations is
proved in~\cite{IkMu}.

According to its definition, $V_0 \subset \R^2$ consists of velocities that produce phase functions with no critical points in the domain of integration.  As a result one can recover polynomial decay (in $x$) of $I(t,x,\eta)$ of any order by repeated integrations by parts.  In fact for solutions of the discrete Klein-Gordon equation the decay satisfies a number of exponential bounds.
\begin{thm}																											\label{thm:exp1}
For every $\mu > 0$ there exists constants 
$0 < v_\mu \leq \frac{1}{\mu}(1 + 2\sqrt{\lambda_1 + \lambda_2}\,\sinh(\mu/2))$ and
$C_\mu < \omega + 2\sqrt{\lambda_1 + \lambda_2}\cosh(\mu/2)$ such that
\begin{align}
\label{eq:LiebRobinsonCos}
\Big|\frac{1}{(2\pi)^2}\int_{[-\pi,\pi]^2} 
\cos(t\,\gamma(k))e^{ik\cdot x}\,dk\Big| &\leq e^{-\mu(|x|-v_\mu|t|)}\\
\Big|\frac{1}{(2\pi)^2}\int_{[-\pi,\pi]^2} \frac{\sin(t\,\gamma(k))}{\gamma(k)}
e^{ik\cdot x}\,dk\Big| &\leq e^{-\mu(|x|-v_\mu|t|)} \label{eq:LiebRobinsonSin}\\
\Big|\frac{1}{(2\pi)^2}\int_{[-\pi,\pi]^2} 
\gamma(k) \sin(t\,\gamma(k))e^{ik\cdot x}\,dk\Big| 
&\leq C_\mu e^{-\mu(|x|-v_\mu|t|)}  \label{eq:LiebRobinsonQtoP}
\end{align}
\end{thm}

%

The upper bound for $v_\mu$ as stated in Theorem~\ref{thm:exp1} behaves as
expected for large $\mu$~(see Corollary~2.2 in \cite{nachtergaele2009})
but it has some evident
drawbacks over the rest of the range.  First, the sharp value of $v_\mu$ must
be an increasing function of $\mu$ so the apparent asymptote as $\mu\to 0$
is an artifact of the calculation.  In addition the estimates~\eqref{eq:LiebRobinsonCos}
and~\eqref{eq:LiebRobinsonSin} don't show any time-decay when applied to a point
$x \in tV_0$ with $\frac{|x|}{t} \leq v_\mu$.  The last result shows that in fact every
$x \in tV_0$ is subject to an effective exponential bound.

\begin{thm}																											\label{thm:exp2}
Let $V_0$ be the set defined in \eqref{eq:setsV}.  For any $x \in \R^2$ with
$\frac{x}{t} \in V_0$ there exists $\mu > 0$ and constants $C_1 < \infty$,
$C_2 \leq \sqrt{\omega^2 + 4(\lambda_1 + \lambda_2)}$ such that
\begin{align}
\label{eq:ExponentialCos}
\Big|\frac{1}{(2\pi)^2}\int_{[-\pi,\pi]^2} 
\cos(t\,\gamma(k))e^{ik\cdot x}\,dk\Big| &\leq e^{-\mu\, \dist(x,\, tV_1)}\\
\Big|\frac{1}{(2\pi)^2}\int_{[-\pi,\pi]^2} \frac{\sin(t\,\gamma(k))}{\gamma(k)}
e^{ik\cdot x}\,dk\Big| &\leq C_1 e^{-\mu\, \dist(x,\, tV_1)} \label{eq:ExponentialSin} \\
\Big|\frac{1}{(2\pi)^2}\int_{[-\pi,\pi]^2} 
\gamma(k) \sin(t\,\gamma(k))e^{ik\cdot x}\,dk\Big| 
&\leq C_2 e^{-\mu\, \dist(x,\, tV_1)}  \label{eq:ExponentialQtoP}
\end{align}
\end{thm}

\subsection{Bounds in $\ell^p(\Z^2)$ and well-posedness for discrete NLKG}
If one combines Corollary~\ref{thm:unif} and Theorem~\ref{thm:exp1}
with $\mu = 1$, the end result is that each propagator of the discrete
Klein-Gordon equation~\eqref{eq:discKleinGordon} is bounded globally by $|t|^{-3/4}$
and decays exponentially outside of a region whose diameter is proportional to $|t|$.
In addition to these pointwise bounds, the $\ell^2$ norm
of solutions is bounded for all time because of the embedding
$\ell^1(\Z^2) \subset \ell^2(\Z^2)$ and the uniform boundedness of 
functions $\cos(t\gamma(k))$ and $\frac{\sin(t\gamma(k))}{\gamma(k)}$
with respect to $t$ and $k$.

Interpolating between the best values of these estimation methods,
one obtains the bounds
\begin{equation} \label{eq:ell_p} 
\norm{u(\,\cdot\, , t)}_p\leq C|t|^{\frac{3}{2p}-\frac{3}{4}} (\norm{g}_1 + \norm{h}_1)
\quad \text{for\ }p \geq 2.
\end{equation}

There is a well-established route from here to proving global estimates for
the associated nonlinear discrete Klein-Gordon equation with power-law nonlinearity,
\begin{equation} \tag{\ref{eq:discNLKG}}
\left\{
\begin{aligned}
&u_{tt}(x,t) - {\textstyle \sum_{j=1}^2}\lambda_j 
\big(u(x+e_j,t) + u(x-e_j,t) - 2u(x,t)\big)
+ \omega^2 u(x,t) = |u|^{\beta - 1}u \\
&u(x,0) = g(x) \\
&u_t(x,0) = h(x)
\end{aligned} \right.
\end{equation}
the initial data $g, h \in \ell^1(\Z^2)$ are small and $\beta$ is large enough
so that $\norm{u^\beta(\,\cdot\, , t)}_1$ is integrable in time for some $p \leq \beta$.
Based on the time-decay available in~\eqref{eq:ell_p}, one can prove the following.

\begin{prop} \label{prop:discNLKG}
Given $\beta > \frac{10}{3}$, there exists $\eps > 0$ and $C < \infty$ 
such that for all initial data with $\norm{g}_1 + \norm{h}_1 < \eps$,
the equation~\eqref{eq:discNLKG} has a unique solution $u(x,t)$ satisfying
\begin{equation*}
\norm{u(\,\cdot\,,t)}_2 \leq C(\norm{g}_1 + \norm{h}_1)
\quad \text{and} \quad
\norm{u(\,\cdot\,,t)}_\infty \leq C|t|^{-\frac34}(\norm{g}_1 + \norm{h}_1)
\end{equation*}
\end{prop}

In brief, one treats the nonlinearity as an inhomogeneous term and shows that
iteration of Duhamel's formula
\begin{equation} \label{eq:Duhamel}
u(\,\cdot\,,t) = \cos(t\sqrt{H})g + \frac{\sin(t\sqrt{H})}{\sqrt{H}}h 
+ \int_0^t \frac{\sin((t-s)\sqrt{H})}{\sqrt{H}} \big(|u|^{\beta-1}u(\,\cdot\,, s)\big)\,ds
\end{equation}
yields a unique fixed point in the space of functions for which
$u \in L^\infty_t\ell^2_x$ and $(1+|t|)^{3/4}u(t,x)$ is bounded.
The condition $\beta >\frac{10}{3}$ is needed here to insure that the norm bound
\begin{equation*}
\|u^\beta (\,\cdot\,, s)\|_1 \leq \norm{u(\,\cdot\,,t)}_2^2
\norm{u(\,\cdot\,,s)}_\infty^{\beta-2} \lesssim (1+|s|)^{(6-3\beta)/4}
\end{equation*}
is integrable with respect to $s$.

For the first two terms of~\eqref{eq:Duhamel},
which do not depend on $u$, Corollary~\ref{thm:unif}
provides a bound in $\ell^\infty(\Z^2)$ with time decay $|t|^{-3/4}$.
The uniform $\ell^2$-bound in for these terms
holds by embedding $\ell^1(\Z^2) \subset
\ell^2(\Z^2)$, and observing that the Fourier multipliers
$\cos(t \gamma(k))$ and $\gamma^{-1}(k)\sin(t\gamma(k))$ are smaller
than $1$ and $\omega^{-1}$ respectively for all $t > 0$.

For the nonlinear term, one has the two operator estimates:
\begin{equation*}
\Big\Vert\frac{\sin(t-s)\sqrt{H}}{\sqrt{H}}f\Big\Vert_\infty 
\leq C|t-s|^{-\frac34}\norm{f}_1 
\quad \text{and} \quad
\Big\Vert\frac{\sin(t-s)\sqrt{H}}{\sqrt{H}}f\Big\Vert_2
\leq \omega^{-1} \norm{f}_1.
\end{equation*}
Then the main integral inequality needed is
\begin{equation*}
\int_0^t |t-s|^{-\frac34}(1+|s|)^{(6-3\beta)/4}\,ds \leq C(1+|t|)^{-\frac34}
\quad \text{for $\beta > \frac{10}{3}$.}
\end{equation*}
Since $\beta > 1$, the mapping $u \to |u|^{\beta-1}u$ has small Lipschitz constant
if $u$ is sufficiently small.  If the initial values $f$ and $g$ are similarly small,
the rest of the argument to establish a contractive mapping 
in~\eqref{eq:Duhamel} is standard.

\begin{rem}
The exponent in~\eqref{eq:ell_p} is probably not optimal.  It is conjectured
in~\cite{mielkepatz} that the decay rate should be $|t|^{-\alpha_p}$,
where $\alpha_p = \min(\frac{p-2}{p}, \frac{3p-5}{4p})$.  In such a
case, a version of Proposition~\ref{prop:discNLKG} would be true
for all $\beta > 3$, with an additional time decay estimate 
$\norm{u(\,\cdot\,,t)}_3 \leq C|t|^{-1/3}(\norm{g}_1 + \norm{h}_1)$.
governing the solution.


The conjecture rests on the premise that the pointwise bounds~\eqref{eq:poly34}
and~\eqref{eq:poly56}
apply inside regions of area $|t|^{5/4}$ and $|t|^{4/3}$ respectively,
whereas our statement of Theorem~\ref{thm:poly} indicates larger regions with
area $(\delta t)^2$ and $\delta t^2$ instead.  Refinements of Theorem~\ref{thm:poly}
in a neighborhood of $tV_3$ and $tV_2$ appear consistent with
properties of local coordinate transformations that resolve the degenerate
oscillatory integrals.  It seems reasonable to believe that explicit construction
of the coordinate transformations, based on calculations in
Sections~\ref{sec:proofs} and~\ref{sec:phasefunc}, would
lead to a complete verification of the bounds in~\cite{mielkepatz}.
\end{rem}

\subsection{Remarks on the discrete wave equation ($\omega = 0$)}
\label{sec:wave}
Many of the implicit constants in Theorem~\ref{thm:poly} and its corollaries,
in particular~\eqref{eq:poly56} and~\eqref{eq:global34},
grow without bound as $\omega$ decreases to zero.
Such behavior occurs because when $\omega$ vanishes, 
the phase function $\gamma$ ceases to
be analytic at the origin, instead developing a singularity of the form
\begin{equation*}
\gamma_0(k) = \big(2\lambda_1(1-\cos k_1) + 2\lambda_2(1-\cos k_2)\big)^{1/2}
= \sqrt{\lambda_1k_1^2 + \lambda_2k_2^2} + O(|k|^3).
\end{equation*}
Meanwhile the curve of $K_2$ (see Lemma~\ref{lem:Kpicture}) winding around
the origin contracts to this one point in the $\omega\searrow 0$ limit.  
At this point the velocity map  $\mathcal{V} = \nabla \gamma_0$
is bounded but not continuous.
Its values are
\begin{equation*}
\nabla\gamma_0(k) = T\bigg(\frac{Tk}{|Tk|}\bigg) + O(|k|^2) \text{ for } 
|k| \ll 1,
\end{equation*}
where $T$ is the diagonal matrix with entries $\sqrt{\lambda_1}$ and 
$\sqrt{\lambda_2}$ respectively.  Note that for all nonzero $k$
the leading order expression
$T(Tk/|Tk|)$ lies on the
ellipse with semiaxis lengths $\sqrt{\lambda_j}$, which is the light cone for
the analogous wave equation on $\R^2$.

A secondary concern affecting Corollary~\ref{cor:poly} is that the auxiliary
function $\eta = \gamma_0^m$ is not smooth, and in fact is unbounded for
the choice $m = -1$ corresponding to the propagator $\sin(t\sqrt{H})/\sqrt{H}$.

Schultz~\cite[Section 3]{schultz98} provides a detailed analysis of the
light-cone behavior for the wave equation when $\lambda_1 = \lambda_2 = 1$.
The leading order term for the sine propagator is a fractional integral of
the Airy function
\begin{equation*}
I(t, vt, \gamma_0^{-1}) \sim \frac{C\sqrt{h(v)}}{t^{2/3}}
\int_0^\infty \frac{{\rm Ai}(z - h(v)(1-|v|)t^{2/3})}{\sqrt{z}}\,dz, \ 
\text{ provided }\  \big|1-|v|\big| \ll t^{-1/2},
\end{equation*}
where $h(v)$ is a smooth function that is not radially symmetric but depends meaningfully on
the direction of $v$.

We claim that the methods in~\cite{schultz98} apply to the more general case
$\lambda_1, \lambda_2 > 0$, $\omega = 0$ with minimal modification.
Specifically, the leading order expression will be
\begin{equation} \label{eq:Airy}
I(t, vt, \gamma_0^{-1}) \sim \frac{C\sqrt{\tilde{h}(v)}}{t^{2/3}}
\int_0^\infty \frac{{\rm Ai}(z - \tilde{h}(v)(1-|T^{-1}v|)t^{2/3})}{\sqrt{z}}\,dz, \ 
\text{ provided }\  \big|1-|T^{-1}v|\big| \ll t^{-1/2}.
\end{equation}
The profile of $\tilde{h}$ depends on the chosen values of $\lambda_1, \lambda_2$,
and in particular the ratio $\lambda_1/\lambda_2$.  It is known from the size and
oscillation properties of the Airy function that 
$\big|\int_0^\infty  z^{-1/2}{\rm Ai}(z-y)\,dz \big| \leq C(1+|y|)^{-1/2}$, leading to
following local (and global) bound.
\begin{prop}
Fix $\omega =0$ and $\lambda_1, \lambda_2 > 0$.  There exists $C < \infty$ such
that for $x \in \Z^2$,
\begin{equation}
|I(t,x,\gamma_0^{-1})| \leq \frac{C}{|t|^{2/3}}
\end{equation}
with the maximum values occurring close to the light cone $\{|T^{-1}x| = |t|\}$.
Spatial decay in the vicinity of the light cone follows the bound
\begin{equation}
|I(t,x,\gamma_0^{-1})| \leq \frac{C}{|t|^{2/3}\big(1+\big||t|-|T^{-1}x|\big|t^{-1/3}\big)^{1/2}}.
\end{equation}
\end{prop}

The location of the four cusps that constitute $V_3$ is relatively easy to
determine for the discrete wave equation once one has built up the
computational structures in Sections~\ref{sec:proofs} and~\ref{sec:phasefunc}.
Cusps occur when the algebraic expressions~\eqref{funcF}
and~\eqref{eq:Gtilde} both vanish at a nontrivial point $(a,b)$.
(The notation $a = \cos k_1$ and $b = \cos k_2$ is introduced in 
Section~\ref{sec:phasefunc} in an effort to streamline calculations.)
If $\omega = 0$ then \eqref{funcF} and~\eqref{eq:Gtilde} are
homogeneous linear functions of $\lambda_1$ and $\lambda_2$, 
so a nontrival simultaneous solution is possible only if the two expressions are
in fact linearly dependent.  After stripping away spurious factors from the determinant
one is left with the relation
\begin{equation*}
\cos k_2 = \frac{-\cos k_1}{1+2\cos k_1}.
\end{equation*}
Plugging this back into the equations yields that $\cos k_1$ must be the unique root of the cubic equation
$\lambda_1(1-a)^2(1+2a) = \lambda_2(1+3a)^2$ lying in the interval $-\frac13 < a < 1$.
When the velocity function $\nabla \gamma$ is evaluated at this special point $(k_1, k_2)$ the
result is as follows.
\begin{prop}
Fix $\omega = 0$ and $\lambda_1, \lambda_2 > 0$.  Let $a^*$ be the unique solution of
\begin{equation*}
\lambda_1(1-a)^2(1+2a) = \lambda_2(1+3a)^2, \quad -\frac13 < a < 1,
\end{equation*} and let $b^* = \frac{-a^*}{1+2a^*}$.  Then the point of $V_3$ in
the first quadrant has coordinates 
\begin{equation}
(v_1, v_2) = \bigg(\Big(\frac{\sqrt{1+3a^*}}{2}\Big)\lambda_1^{1/2}, 
 \Big(\frac{\sqrt{1+3b^*}}{2}\Big)\lambda_2^{1/2}\bigg).
\end{equation}
\end{prop}

\section{Proof of the main results}  \label{sec:proofs}

\subsection{Proof of Theorem \ref{thm:poly}}						\label{sec:polyproof}
The material of this section is presented in the following order: we start with some background information on oscillatory integrals, followed by the main local results (Lemma~\ref{lm:heightk} and Corollary~\ref{cr:OmegakCk}), and the proof of Theorem \ref{thm:poly} is then obtained from local estimates through a partition of unity argument.

Consider an oscillatory integral in several variables
\begin{equation}
I(t, \eta) = \int_{\R^d} e^{it \phi(k)}\eta(k) dk,
\end{equation}
where $\eta$ is supported in a neighborhood of an isolated critical point $k^*$ of $\phi$ (we follow the notation introduced in \cite{Var} and  \cite{greenblatt}). When it is convenient to do so, one may apply an affine translation so that the critical point is located at the origin. If $\supp (\eta)$ is small enough, $I(t, \eta)$ has an asymptotic expansion
\begin{equation} \label{eq:expansion}
I(t, \eta) \approx e^{it\phi(0)} \sum_{j=0}^{\infty} (d_j(\eta) + d_j'(\eta)\ln(t)) t^{-s_j}, \quad \text{ as } t \to \infty,
\end{equation}
where $s_j$ is an increasing arithmetic progression of positive rational numbers independent of $\eta$. 
The oscillatory index of the function $\phi$ at $k^*$ is defined to be the leading-order exponent $s_0$. We assume that $s_0$ is chosen to be minimal such that in any sufficiently small neighborhood $U$ containing $k^*$ either $d_0(\eta)$ or $d_0'(\eta)$ is nonzero for some $\eta$ supported in $U$.

Estimates~\eqref{eq:poly34} and~\eqref{eq:poly56} are essentially statements about the oscillatory index of
$\phi_v(k)$ at its critical points for different values of $v$.  The following algorithm assists in their computation. 

Suppose $\phi$ is analytic with a critical point at $k^*$.  Locally there is a Taylor series expansion $\phi(k) = c_0 + \sum c_n (k-k^*)^n$, with the sum ranging over all $n \in \Z^d_+$ with $n_1 + \ldots + n_d \geq 2$.  
Let $K \subset \Z^+$ be the collection of all indices $n$ for which $c_n \not= 0$.

Newton's polyhedron associated to $\phi$ at its critical point $k^*$ is defined as the convex hull of the set
\begin{equation*}
\bigcup_{n \in K} (n + \R_+^d),
\end{equation*}
where $\R_+^d$ is the positive octant $\{x \in \R^d: x_j \ge 0 \text{ for } 1 \le j \le d\}$.
We denote Newton's polyhedron of $\phi$ by $N_+(\phi)$.
Newton's diagram of $\phi$ is the union of all compact faces of $N_+(\phi)$. 
%
Finally, the Newton distance $d(\phi)$ is defined as $d(\phi) = \inf\{t: (t,t) \in N_+(\phi)\}$.

Note that the vanishing of Taylor coefficients is affected by changes to the underlying coordinates, thus each local coordinate system $y$ generates its own sets $K^y$ and $N_+^y(\phi)$ and Newton distance $d^y(\phi)$.  Define the height of an analytic function $\phi$ at its critical point to be $h(\phi) := \sup \{d^y(\phi)\}$,
with the supremum taken over all local coordinate systems $y$.  A coordinate system $y$ is called adapted if $d^y(\phi) = h(\phi)$.

It was shown in \cite{Var} (p. 177, Theorem~0.6) that under some natural assumptions, the oscillation index $s_0$ of a function $\phi$ is equal to $1/h(\phi)$.

We now compute the height of the phase of the integral $I(t, x, \eta)$ at its critical point(s). Recall that with the notation $v = \frac x t$, the phase of $I(t, x, \eta)$ can be written in the form \eqref{eq:phasephiv}. Given a point $k^*\in [-\pi, \pi]^2$, the value $v = \nabla\gamma(k^*) \in \R^2$ is the unique choice for which $\phi_v$ has a critical point at $k^*$.
%
Denote the height of this $\phi_{v}$ at $k^*$ by $h(k^*)$.

\begin{lem}																				\label{lm:heightk}
The height function $h(k^*)$, defined above, is constant on each of the sets $\{K_i\}_{i=1}^3$ defined
in~\eqref{eq:setK} with the following values: 
\begin{enumerate}
\item if $k^* \in K_1$, then $h(k^*) = 1$,
\item if $k^* \in K_2$, then $\displaystyle h(k^*) = 6/5$,
\item if $k^* \in K_3$, then $\displaystyle h(k^*) = 4/3$.
\end{enumerate}
\end{lem}
\begin{proof}
Note that $\gamma(k)$ and $\phi_v(k)$ differ by a linear function, so their derivatives coincide except at the
first order.
In the simpler case $k^* \in K_1$,  $\det D^2 \gamma(k^*) = \det D^2\phi_v(k^*) \ne 0$ by definition.
Moreover, the determinant of the Hessian of $\phi_{v}$ at $k^*$ remains non-zero in any local coordinate system,
thus $h(k^*) = d/2 = 1$.

Next consider $k^* \in K_2$. In this case $\det D^2 \gamma(k^*) = 0$ and $(\xi \cdot \nabla)^3\gamma(k^*) \ne 0$, where $\xi$ is an eigenvector of $D^2 \gamma(k^*)$ corresponding to the zero eigenvalue.  The mixed second-order derivative vanishes because $D^2\gamma(k^*)$ has orthogonal eigenvectors, and by Proposition~\ref{prop:rank1} it is guaranteed that $(\xi^\perp\cdot \nabla)^2 \gamma(k^*) \ne 0$.  
This information suffices to compute the Newton's distance of $\phi_v$ at $k^*$ in the linear coordinate system with axes
$\{\xi^\perp, \xi\}$ and given by coordinates $k - k^* = y_1 \xi^\perp + y_2 \xi$.  

The associated Newton's polyhedron is of the form displayed on Figure~\ref{pic:NP2}, with the Newton's distance being $\displaystyle 6/5$. 
\begin{figure}[ht]
\begin{picture}(100,100)
\put(0,20){\vector(1,0){105}}											
\put(110,18){$n_1$}
\put(20,0){\vector(0,1){100}}
\put(6,95){$n_2$}
\put(20,19){\circle{4}}														
\put(40,19){\circle{4}}	
\put(20,40){\circle{4}}		
\put(20,60){\circle{4}}
\put(40,40){\circle{4}}
\put(57,17){$\times$}                             
\put(17,77){$\times$}
\thicklines
\qbezier(60,20)(60,20)(20,80)								    	
\qbezier(60,20)(60,20)(100,20)
\qbezier(20,80)(20,80)(20,100)
\put(60,70){$N(\phi_v)$}                          
\thinlines
\qbezier(25,95)(25,95)(100,95)											
\qbezier(25,85)(25,85)(100,85)
\qbezier(30,75)(30,75)(58,75)
\qbezier(92,75)(90,75)(100,75)
\qbezier(37,65)(37,65)(100,65)									
\qbezier(43,55)(43,55)(100,55)	
\qbezier(50,45)(50,45)(100,45)	
\qbezier(56,35)(56,35)(100,35)	
\qbezier(62,25)(62,25)(100,25)
\qbezier[20](20,20)(30,30)(44,44)									
\qbezier[20](44,44)(44,34)(44,20)	
\put(46,9){{\small $\frac 6 5$}}												
\end{picture}
\caption{Newton's polyhedron and Newton's distance of the Taylor series corresponding to $\phi_v$ for $k^*\in K_2$}
\label{pic:NP2}
\end{figure}
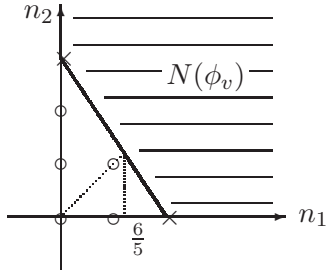

In order to show that $\{y_1, y_2 \}$ is an adapted coordinate system, and therefore $h(k^*) = 6/5$, we use the following result from Varchenko:
\begin{prop}[{\cite[part 2 of Proposition~0.7]{Var}}]  \label{prop:Var07}
Assume that for a given series $f = \sum c_n y^n$, the point $(d(f),d(f))$ lies on a closed compact face $\Gamma$ of the Newton's polyhedron. Let $a_1 n_1 + a_2 n_2 = m$ be the equation of the straight line on which $\Gamma$ lies, where $a_1$, $a_2$, and $m$ are integers and $a_1$ and $a_2$ are relatively prime. Then the coordinate system $y$ is adapted if both numbers $a_1$ and $a_2$ are larger than $1$. 
\end{prop}
The Newton polyhedron displayed on Figure~\ref{pic:NP2} has only one compact face (that also contains the point $(d(\phi_v), d(\phi_v))$), which lies on the line with the equation $3n_1 + 2 n_2 = 6$.  Since $2$ and $3$ are relatively prime, the coordinate system is adapted by Proposition~\ref{prop:Var07}. 

Finally, let $k \in K_3$. To provide a more concise notation,  let $\partial_\xi$ and $\partial_{\xi^\perp}$ indicate the
directional derivatives $\xi \cdot \nabla$ and $\xi^\perp \cdot \nabla$ respectively.  These also serve as partial derivatives $\partial_{y_2}$ and $\partial_{y_1}$ with respect to coordinates $\{y_1, y_2\}$.  By definition of $K_3$ we have
\begin{equation}	\label{eq:xiderK3}
\partial^2_{\xi} \gamma(k^*) = 0, \qquad 
\partial_{\xi} \partial_{\xi^\perp} \gamma(k^*) = 0, \qquad
\partial^2_{\xi^\perp} \gamma(k^*) \ne 0, \quad \text{ and } \quad \partial^3_{\xi} \gamma(k^*) = 0.
\end{equation}
On the other hand, one can show that $\partial^4_{\xi}\gamma(k^*)$ and $\partial^2_{\xi}\partial_{\xi^\perp}\gamma(k^*)$ do not both vanish (see Lemma~\ref{lm:fourthder}). 
The possible Newton's polyhedra that arise are indicated in Figure~\ref{pic:NP3}, however, the Newton distance is equal to $\displaystyle 4/3$ in all situations: 
\begin{figure}[ht]
\begin{picture}(100,120)
\put(0,20){\vector(1,0){105}}											
\put(110,18){$n_1$}
\put(20,0){\vector(0,1){120}}
\put(6,115){$n_2$}
\put(20,20){\circle{4}}														
\put(40,20){\circle{4}}	
\put(20,40){\circle{4}}		
\put(20,60){\circle{4}}
\put(40,40){\circle{4}}
\put(20,80){\circle{4}}
\put(40,60){\circle{4}}
\put(20,100){\circle{4}}
\put(57,17){$\times$}                             
\put(17,97){$\times$}
\put(37,57){$\times$}
\thicklines
\qbezier(60,20)(60,20)(40,60)								    	
\qbezier[20](40,60)(30,80)(20,100)
\qbezier[20](40,60)(40,60)(20,120)
\qbezier(60,20)(60,20)(100,20)
\put(60,70){$N(\phi_v)$}                          
\thinlines
\qbezier(25,115)(25,115)(100,115)											
\qbezier(25,105)(25,105)(100,105)
\qbezier(28,95)(28,95)(100,95)
\qbezier(33,85)(33,85)(100,85)
\qbezier(38,75)(38,75)(58,75)
\qbezier(92,75)(90,75)(100,75)
\qbezier(43,65)(43,65)(100,65)									
\qbezier(48,55)(48,55)(100,55)	
\qbezier(53,45)(53,45)(100,45)	
\qbezier(58,35)(58,35)(100,35)	
\qbezier(63,25)(63,25)(100,25)
\qbezier[20](20,20)(33,33)(46,46)									
\qbezier[20](46,46)(46,34)(46,20)	
\put(47,9){{\small $\frac 4 3$}}												
\end{picture}
\caption{Possible Newton's polyhedra and Newton's distance of the Taylor series corresponding to $\phi_v$ for $k^*\in K_3$}
\label{pic:NP3}
\end{figure}
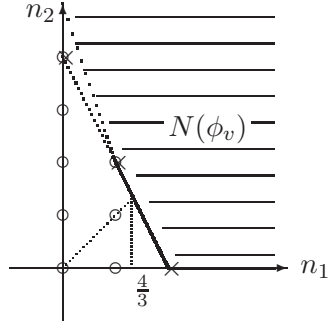

Moreover,  in all situations the face containing the point $(d(\phi_v), d(\phi_v))$ lies on the line with the equation $2n_1  + n_2 = 4$,  and Proposition~\ref{prop:Var07} does not apply. To verify that  the system $\{y_1, y_2\}$ is adapted in this case, we use a different result from Varchenko:
\begin{prop}[{\cite[Proposition~0.8]{Var}}]				\label{prop:Var08}
Assume that for a given series $f = \sum c_n y^n$, the point $(d(f),d(f))$ lies on a closed compact face $\Gamma$ of the Newton's polyhedron. Let $a_1n_1 +  n_2 = m$ be the equation of the straight line on which $\Gamma$ lies, where $a_1$ and $m$ are integers. Let 
\begin{equation}
f_{\Gamma}(y) = \sum_{n \in \Gamma} c_n y^n
\quad \text {and } \quad
P(y_1) = f_{\Gamma}(y_1, 1).
\end{equation}
If the polynomial $P$ does not have a real root of multiplicity larger than $m(1+a_2)^{-1}$, then $y$ is a coordinate system adapted to $f$.
\end{prop}
For the face $\Gamma$, displayed on Figure~\ref{pic:NP3}, we have 
\begin{equation} \begin{aligned}
f_{\Gamma}(y) &= \frac{\partial^2_{\xi^\perp}\gamma(k^*)}{2} y_1^2  + \frac{\partial^2_{\xi}\partial_{\xi^\perp}\gamma(k^*)}{2}\, y_1y_2^2  + 
\frac{\partial^4_{\xi}\gamma(k^*)}{24}\, y_2^4,
\\
P(y_1) &= \frac{\partial^2_{\xi^\perp}\gamma(k^*)}{2} y_1^2 
+ \frac{\partial^2_{\xi}\partial_{\xi^\perp}\gamma(k^*)}{2}\, y_1
+ \frac{\partial^4_{\xi}\gamma(k^*)}{24}.
\end{aligned}
\end{equation}
The discriminant of $P$ is 
\begin{equation}
\mathcal{D} = \left(\frac{\partial^2_{\xi}\partial_{\xi^\perp}\gamma(k^*)}{2}\right)^2 - 4\frac{\partial^4_{\xi}\gamma(k^*)}{24} \frac{\partial^2_{\xi^\perp}\gamma(k^*)}{2} = \frac 1 {12} \left(3(\partial^2_{\xi}\partial_{\xi^\perp}\gamma(k^*))^2 - \partial^2_{\xi^\perp}\gamma(k^*) \partial^4_{\xi}\gamma(k^*)\right). 
\end{equation}
It follows from Lemma~\ref{lm:fourthder} that this discriminant is nonzero whenever $k^* \in K_3$, thus $P$ can have real roots of multiplicity at most one. On the other hand, $m(1+a_2)^{-1} = 4/3$, and by Proposition~\ref{prop:Var08} the coordinate system $\{y_1, y_2\}$ is adapted.
\end{proof}

As was stated earlier, the height of the phase function determines the decay order of an oscillatory integral in a neighborhood of its critical point. In \cite{Var} it is shown that the oscillation index of a phase $\phi$ is equal to $1/h(\phi)$, giving both upper and lower bounds for the decay rate.  More recently, Ikromov and M\"uller in \cite{IkMu} showed that the upper bound is stable under linear perturbations of the phase function. Their result, combined with Lemma~\ref{lm:heightk}, brings us the following
\begin{cor}																				\label{cr:OmegakCk}
Let sets $\{K_i\}_{i=1}^3$ be defined by \eqref{eq:setK} and fix $k^* \in [-\pi, \pi]^2$. Then there exist a neighborhood of $k^*$, $\Omega_{k^*}$, and a positive constant $C_{k^*}$ such that for all $\eta$ supported in $\Omega_{k^*}$,
\begin{align}
\label{eq:local34}\left|\tilde{I}(t, x, \eta)\right| &\le C_{k^*} \norm{\eta}_{C^3(\R^2)} \frac 1 {|t|^{3/4}},
 \quad \text{ if } k^* \in K_3,\\
\label{eq:local56}\left|\tilde{I}(t, x, \eta)\right| &\le C_{k^*} \norm{\eta}_{C^3(\R^2)} \frac 1 {|t|^{5/6}},
\quad \text{ if } k^* \in K_2,\\ 
\label{eq:local1}\left|\tilde{I}(t, x, \eta)\right| &\le C_{k^*} \norm{\eta}_{C^3(\R^2)} \frac 1 {|t|},
\quad \quad \text{ if } k^* \in K_1.
\end{align} 
for all $x \in \R^2$.
\end{cor}
\begin{proof}
By a direct consequence of Theorem 1.1 in \cite{IkMu}, for a point $k^* \in [-\pi, \pi]^2$, there exist a neighborhood $\Omega_{k^*}$ and a positive constant $C_{k^*}$ such that
\begin{equation}
\left|\int_{\Omega_{k^*}} e^{it (\phi_{v}(k) + x\cdot k)} \eta(k) \, dk\right| \le C_{k^*} \norm{\eta} \frac 1 {|t|^{h(k^*)}},
\end{equation}
for all $x \in \R^2$ and $\eta$ supported in $\Omega_{k^*}$. This, together with the result of Lemma~\ref{lm:heightk}, proves the claim.
\end{proof}
We can extract some additional important information about the neighborhoods $\Omega_{k^*}$, introduced in Corollary~\ref{cr:OmegakCk}, that will be useful for the main part of the proof. Specifically, if $k^* \in K_1$ then the corresponding neighborhood $\Omega_{k^*}$ does not contain any points from $K_2 \cup K_3$, and if $k^* \in K_2$ then $\Omega_{k^*}$
is disjoint from $K_3$.  To prove the latter claim, suppose there is a point $k_0 \in K_3$ that also belongs to $\Omega_{k^*}$.
Then with $v = \nabla \gamma(k_0)$ the oscillation index of $\phi_v$ in a neighborhood of $k_0$ is equal to 3/4.  For $\eta$
supported in a small neighborhood of $k_0$ inside of $\Omega_{k^*}$, and $x = tv$, the asymptotic lower bound dictated by~\eqref{eq:expansion} and statement (3) of Lemma~\ref{lm:heightk} contradicts the decay rates of~\eqref{eq:local56} and~\eqref{eq:local1}.


At last, we need the following well-known estimate.
\begin{lem}												\label{lm:fastpolydecay}
If $\supp \eta$ does not contain any critical points of the phase function $k \cdot x - t\gamma(k)$ then for any $M > 0$,
\begin{equation}
|\tilde{I}(t,x,\eta)| \le C(M,\eta, d)\frac {1}{|t|^M}.
\end{equation}
where $d$ is the infimum of $|x/t - \nabla \gamma(k)|$ over the support of $\eta$.
\end{lem}
\begin{proof} [Proof of Theorem \ref{thm:poly}.]
Fix $\delta > 0$. The nonstationary phase bound~\eqref{eq:polyN} follows immediately from the construction, as the
gradient of $x\cdot k - t\gamma(k)$ must have magnitude at least $\dist(x, tV_1)$.

To prove \eqref{eq:poly34}, Take the system of neighborhoods $\{\Omega_k\}_{k \in [-\pi, \pi]^2}$ described in Corollary~\ref{cr:OmegakCk}. By construction $\{\Omega_k\}_{k \in [-\pi, \pi]^2}$ covers $[-\pi, \pi]^2$ and we can choose a finite sub-cover, say $\{\Omega_{j}\}_{j=1}^{N_0}$. Now, let a collection of smooth functions $\{\omega_j\}_{j=1}^{N_0}$ form a partition of unity with respect to $\{\Omega_j\}_{j=1}^{N_0}$, then 
\begin{equation}				\label{eq:partunityineq}							
|I(t, x, \eta)| \le \sum_{j} |I(t, x, \eta_j)| = \sum_j |\tilde{I}(t, x, \eta_j)|,
\end{equation}
where $\eta_j = \eta \, \omega_j$, is supported in $\Omega_j$. Since for $t$ away from  zero every integral that satisfies  \eqref{eq:local56} or \eqref{eq:local1} also satisfies \eqref{eq:local34}, and all three are uniformly bounded for all times, we have
\begin{equation} 																												\label{eq:34proof}
|I(t, x, \eta)| \le \sum_{j} C_j \norm{\eta_j}_{C^3(\R^2)} \frac 1 {|t|^{3/4}} = C(\eta) \frac 1 {|t|^{3/4}}.
\end{equation}
Note that even though \eqref{eq:34proof} holds for all $x \in \R^2$, better estimates are available when $x$ is removed from $tV_3$.

To prove the estimates~\eqref{eq:poly56} and~\eqref{eq:poly1} we need to refine our construction of the cover so that the
$|t|^{-3/4}$ bound in~\eqref{eq:local34} is never invoked (in the latter case one should also avoid applying~\eqref{eq:local56}). The following construction will suit both situations. 

The function $\mathcal{V}$, defined by \eqref{eq:nablagammamap}, is uniformly continuous on $[-\pi, \pi]^2$, so we can choose an $0 < \epsilon = \epsilon(\delta) < \pi/2$ such that $$\diam(\mathcal{V}(B_{\epsilon})) < \delta/2$$ for every ball $B_{\epsilon}$ of radius $\epsilon$. At each $k \in [-\pi, \pi]^2$ define a smaller $\delta$-dependent neighborhood
\begin{equation}
\Omega_k(\delta) = \Omega_k \cap B_{\epsilon}(k),
\end{equation}
where $\Omega_k$ is again as in Corollary~\ref{cr:OmegakCk}. As before, pick a finite sub-collection of $\{\Omega_k(\delta)\}_{k \in [-\pi, \pi]^2}$ that is also a cover of $[-\pi, \pi]^2$, say $\{\Omega_{k_j}(\delta)\}_{j=1}^N$, and generate a partition of unity $\omega_j$ subordinate to this cover. For simplicity of notation we will write $\Omega_j = \Omega_{k_j}(\delta)$ with $j \in \{1, 2, \ldots , N\}$.

Sort the neighborhoods $\Omega_j$ according to the location of their 
"center" point $k_j$.  For each $m = 1, 2, 3$ let 
$J_m := \{j \in \{1, \ldots, N\}: k_j \in K_m\}$, where $K_m$ are the sets
defined in~\eqref{eq:setK}.  The discussion following
Corollary~\ref{cr:OmegakCk} indicates that 
\begin{equation}
\bigg(\bigcup_{j \in J_1\cup J_2} \Omega_j\bigg) \cap K_3 = \emptyset 
\quad \text{ and } \quad 
\bigg(\bigcup_{j \in J_1} \Omega_j\bigg) \cap K_2 = \emptyset.
\end{equation} 

Suppose $x \in \Z^2$ is chosen so that $\dist(x, tV_3) > t\delta$.
In other words, $|x/t - \nabla \gamma(k^*)| > \delta$ for any $k^* \in K_3$.
Moreover $|x/t - \nabla \gamma(k) | > \delta/2$ for all 
$k \in \bigcup_{j \in J_3} \Omega_j$ because each neighborhood has radius
at most $\epsilon$.  Split the sum~\eqref{eq:partunityineq} into two parts
\begin{equation}										\label{eq:56initial}
|I(t, x, \eta)| \le \sum_{j \in J_3} |\tilde{I}(t, x, \eta_j)| + \sum_{j\notin J_3} |\tilde{I}(t, x, \eta_j)|.
\end{equation}
Lemma~\ref{lm:fastpolydecay} applies to each term in the first sum,
with $d = \delta/2$.  Terms in the second sum are bounded by~\eqref{eq:local56} or~\eqref{eq:local1}.
The slowest time-decay out of these has the rate $|t|^{-5/6}$ from~\eqref{eq:local56}, which
verifies~\eqref{eq:poly56}.

The argument is similar if $\dist(x, t(V_3 \cup V_2)) > t\delta$.
One splits~\eqref{eq:partunityineq} in the parts
\begin{equation}										\label{eq:1initial}
|I(t, x, \eta)| \le \sum_{j \in J_2\cup J_3} |\tilde{I}(t, x, \eta_j)| + \sum_{j\notin J_2 \cup J_3} |\tilde{I}(t, x, \eta_j)|,
\end{equation}
and once again Lemma~\ref{lm:fastpolydecay} applies to each term in the first 
sum, with $d = \delta/2$, and terms in the second sum are bounded by~\eqref{eq:local1}.
This is sufficient to verify~\eqref{eq:poly1}, completing the proof of Theorem~\ref{thm:poly}.
\end{proof}

\subsection{Proof of Exponential Bounds}
\begin{proof}[Proof of Theorem~\ref{thm:exp1}] 
Note that $\gamma^2(k)$ extends to a complex-analytic function on $k \in \bC^2$
that is periodic under the shifts $k_j \to k_j + 2\pi$, $j = 1,2$.  After composition with
the holomorphic map $\cos(t \sqrt{z})$, the same is true of $\cos(t\,\gamma(k))$.
By shifting the contour of integration for $k_1$ and $k_2$,
the left-hand quantity in~\eqref{eq:LiebRobinsonCos} is equal to
\begin{align*}
\frac{e^{-\mu|x|}}{(2\pi)^2}
\Big| \int_{[-\pi,\pi]^2} \cos(t\,\gamma(k + i\mu{\textstyle \frac{x}{|x|}}))
e^{ik\cdot x}\,dk\Big| 
&\leq \max_{k\in[-\pi,\pi]^2} 
\big|\cos(t\,\gamma(k + i\mu{\textstyle \frac{x}{|x|}}))\big| e^{-\mu|x|} \\
&\leq \max_{k\in[-\pi,\pi]^2} 
e^{|\mathrm{Im}\, t\,\gamma(k + i\mu{\textstyle \frac{x}{|x|}})|}e^{-\mu|x|} \\
&= e^{-\mu(|x|-v_\mu |t|)}
\end{align*}
where $v_\mu = \mu^{-1} \max\{|\mathrm{Im}\,\gamma(k + i\tilde{k})|: 
k \in [-\pi,\pi]^2, |\tilde{k}| = \mu\}$.  Referring back to the definition of $\gamma(k)$
in~\eqref{gammaABk}, one obtains a bound $v_\mu \leq \frac{2}{\mu}
\sqrt{\lambda_1 + \lambda_2}\,\sinh(\mu/2)$ by applying the inequality
$|\mathrm{Im}\,\sqrt{z^2 + w^2}| \leq \sqrt{(\mathrm{Im}\, z)^2 + (\mathrm{Im}\,w)^2}$
for pairs of complex numbers.

The same argument applies to the sine propagator as well, thanks to the bound
$\big| \frac{\sin z}{z}\big| \leq e^{|\mathrm{Im}\,z|}$.  By shifting the integration
contour as above, the left-hand quantity in~\eqref{eq:LiebRobinsonSin} is equal to
\begin{align*}
t \frac{e^{-\mu|x|}}{(2\pi)^2}
\Big| \int_{[-\pi,\pi]^2} \frac{\sin(t\,\gamma(k + i\mu{\textstyle \frac{x}{|x|}}))}
{t\,\gamma(k + i\mu{\textstyle \frac{x}{|x|}})}
e^{ik\cdot x}\,dk\Big| 
&\leq t \max_{k\in[-\pi,\pi]^2} 
e^{|\mathrm{Im}\, t\,\gamma(k + i\mu{\textstyle \frac{x}{|x|}})|}e^{-\mu|x|} \\
&\leq e^{-\mu(|x|-v_\mu |t|)}
\end{align*}
for any  $v_\mu > \mu^{-1} (1+ \max\{|\mathrm{Im}\,\gamma(k + i\tilde{k})|: 
k \in [-\pi,\pi]^2, |\tilde{k}| = \mu\})$.

The computation for~\eqref{eq:LiebRobinsonQtoP} is essentially identical
to~\eqref{eq:LiebRobinsonCos} except that it contains an extra factor of
$\max\{|\gamma(k + i\tilde{k}): |\tilde{k}| = \mu\}$, estimated here by
$\omega + 2\sqrt{\lambda_1 + \lambda_2}\,\cosh(\mu/2)$.

\end{proof}

\begin{proof}[Proof of Theorem~\ref{thm:exp2}]
It will be important to note that $V_0$ is the complement of a convex subset
of the plane. This is stated as part of Proposition~\ref{prop:Vpicture} and
will be proved in Section~\ref{sec:phasefunc}.

For each $\tilde{\mu} \in \R^2$, shifting contours of integration into the
complex plane leads to the bound
\begin{align*}
\Big|\frac{1}{(2\pi)^2}\int_{[-\pi,\pi]^2}  
\cos(t\,\gamma(k))e^{ik\cdot x}\,dk\Big| 
&= \frac{e^{-\tilde{\mu} \cdot x}}{(2\pi)^2}
\Big| \int_{[-\pi,\pi]^2} \cos(t\,\gamma(k + i\tilde{\mu}))
e^{ik\cdot x}\,dk\Big| \\
&\leq \max_{k\in[-\pi,\pi]^2} 
e^{|\mathrm{Im}\, t\,\gamma(k + i\tilde{\mu})|}e^{-\tilde{\mu} \cdot x}
\end{align*}

By assumption $\frac{x}{t}$ lies outside the convex balanced compact set 
$\{\nabla \gamma(k): k \in [-\pi,\pi]^2\} = \overline{V}_1$. 
The complex derivative of $\gamma$ indicates that 
$\mathrm{Im}\,\gamma(k + i\tilde{\mu}) = 
(\nabla\gamma(k))\cdot \tilde{\mu} + o(|\tilde{\mu}|)$, and the
implicit constant in $o(|\tilde{\mu}|)$ converges uniformly across 
$k \in [-\pi,\pi]^2$.
Choose $\mu > 0$ small enough so that 
\begin{equation*}
\big|\mathrm{Im}\,\gamma(k + i\tilde{\mu}) - \tilde{\mu}\cdot (\nabla\gamma(k))\big|
\leq \frac12 \dist({\textstyle \frac{x}{t}}, V_1) |\tilde{\mu}|
\end{equation*}
whenever $|\tilde{\mu}| = 2\mu$.  Then
\begin{align*}
\Big|\frac{1}{(2\pi)^2}\int_{[-\pi,\pi]^2}  
\cos(t\,\gamma(k))e^{ik\cdot x}\,dk\Big| 
&\leq \inf_{|\tilde{\mu}| = 2\mu} \, \max_{k \in [-\pi,\pi]^2}
e^{|(t \nabla \gamma(k))\cdot \tilde{\mu}|}e^{\dist(x,\,tV_1)\mu}
e^{-x\cdot \tilde{\mu}} \\
&= \inf_{|\tilde{\mu}| = 2\mu} \, \max_{k \in [-\pi,\pi]^2}
e^{(t \nabla \gamma(k) - x)\cdot \tilde{\mu} }e^{\dist(x,\, tV_1)\mu} \\
&= e^{-2\mu\, \dist(x,\, tV_1)}e^{\dist(x,\, tV_1)\mu}.
\end{align*}
The first equality follows from the fact that $\nabla \gamma(k)$ is an odd
function, so the absolute value can be optimized with either sign.  The second
equality asserts a geometric principle that given a closed convex set $S$ and
a point $x \not\in S$, 
\begin{equation*}
\sup_{y\in S}\, (y-x) \cdot v \geq -|v|\, \dist(x, S)
\end{equation*}
with equality taking place if $y\in S$ minimizes the distance and $v$ is
parallel to $x-y$.  The argument for the sine propagator is essentially identical
as the extra factor of $t$ that also appears in the proof of
Theorem~\ref{thm:exp1} can be overcome by choosing a slightly smaller
value of $\mu>0$ and introducing a large constant $C_1$.  The value of
$C_2$ is limited by estimating the maximum of $|\gamma(k + i\tilde{k})|$
over $|\tilde{k}| \ll 1$.
\end{proof}

\section{Applications to quantum systems} \label{sec:quantum}
In this section  we apply the main integral estimates from Section~\ref{sec:results} to obtain dispersive estimates in infinite-volume harmonic systems on $\Z^2$. 
For a detailed formal introduction of such systems and their main properties, see \cite{BorSims2012} and references therein.

\subsection{Main Results} 

Recall first a definition of a Weyl algebra over a real linear space $\mathcal{D}$, equipped with a symplectic, non-degenerate
bilinear form $\sigma$. The Weyl algebra over $\mathcal{D}$, which we will denote by $\mathcal{W}( \mathcal{D})$, is
defined to be a $C^*$-algebra generated by Weyl operators, i.e., non-zero elements $W(f)$, 
associated to each $f \in \mathcal{D}$, which satisfy
\begin{equation} \label{eq:invo}
W(f)^* = W(-f) \quad \mbox{for each } f \in \mathcal{D} \, ,
\end{equation}
and
\begin{equation} \label{eq:weylrel}
W(f) W(g) = e^{-i \sigma(f,g)/2} W(f+g) \quad \mbox{for all } f, g \in \mathcal{D} \, .
\end{equation}
Such an algebra with additional properties that $W(0) = \idty$, $W(f)$ is unitary for all $f \in \mathcal{D}$, and $\| W(f) - \idty \| = 2$ for all $ f \in \mathcal{D} \setminus \{0 \}$ is unique up to a $*$-isomorphism (cf. \cite{bratteli1997}, Theorem 5.2.8). See \cite{manuceau1973} and \cite{bratteli1997} for details of Weyl algebra formalism.

Some of the standard choices of $\mathcal{D}$ are $\mathcal{D} = \ell^2( \mathbb{Z}^2)$ or $\mathcal{D} = \ell^1( \mathbb{Z}^2)$ with the symplectic form
\begin{equation} \label{eq:hsig}
\sigma(f,g) = \mbox{Im} \left[ \langle f, g \rangle \right] \, \quad \mbox{for } f, g \in \mathcal{D}.
\end{equation}

The infinite volume harmonic dynamics is introduced by a one-parameter group of $*$-automorphisms $\tau_t$ on $\mathcal{W}( \mathcal{D})$, such that
\begin{equation} \label{eq:quasifree}
\tau_t(W(f))=W(T_t f) \quad \mbox{for all } f \in \mathcal{D},
\end{equation}
where
\begin{align*}
T_tf &=   \cos(t\sqrt H)f + i \frac{\sin(t\sqrt H)}{\sqrt H}\mbox{Re} f - \sqrt H \sin(t\sqrt H)\mbox{Im}f = u_t(x,t) + iu(x, t).
\end{align*}
The function $u$ here is the solution of the discrete Klein-Gordon equation \eqref{eq:discKleinGordon} with initial conditions $g(x) = \mbox{Im} f$ and $h(x) = \mbox{Re} f$.

The commutator norm $\left\| \left[ \tau_t(W(f)), W(g) \right]  \right\|$ is a quantity that measures how fast the dynamics spreads information through the system (specifically, between the supports of $f$ and $g$). Notice that when $t=0$ and the supports of $f$ and $g$ are disjoint, the commutator is zero (in particular, it follows from \eqref{eq:weylrel}). 

 The following notation will be used throughout the rest of this section: let $X = \supp(f)$, $Y = \supp(g)$, $X - Y$ be the difference set
\begin{equation*}
X - Y = \{x - y: x \in X, y \in Y\},
\end{equation*}  
and $B_{r}(S)$ represent an open neighborhood of a set $S$ of radius $r$.

Our first result describes pairs of $W(f)$ and $W(g)$ for which the corresponding commutator norm decays polynomially. There are three possible polynomial regimes;  the choice of the appropriate regime depends on the mutual location of the supports.

\begin{thm} \label{thm:uniest} Let $\tau_t$ be the harmonic dynamics defined as above on $\mathcal{W}( \ell^2( \mathbb{Z}^2))$ and sets $V_i$, $i=2, 3$, be defined by \eqref{eq:setsV}. Then the following statements hold.
\begin{enumerate}
\item There exists a number $C_3 > 0$, such that for all $f,g \in \ell^1( \mathbb{Z}^2)$,
\begin{equation}                                               \label{eq:polyestt34} 
\left\| \left[ \tau_t(W(f)), W(g) \right]  \right\|  \leq  \min \left[2,  \frac{C_3 \| f \|_1 \| g \|_1}{|t|^{3/4}} \right].
\end{equation}
\item For any $\delta>0$, there exists a number $C_2 = C_2(\delta) > 0$, such that for all $f,g \in \ell^1( \mathbb{Z}^2)$ with $X - Y \in \Z^2\setminus B_{t\delta}(tV_3)$,
\begin{align} 
\label{eq:polyestt56} \left\| \left[ \tau_t(W(f)), W(g) \right]  \right\|  &\leq  \min \left[2,  \frac{C_2 \| f \|_1 \| g \|_1}{|t|^{5/6}} \right].
\end{align}
\item For any $\delta>0$, there exists a number $C_1 = C_1(\delta) > 0$, such that for all $f,g \in \ell^1( \mathbb{Z}^2)$ with $X - Y \in \Z^2\setminus B_{t\delta}(t(V_2 \cup V_3))$,
\begin{align}
\label{eq:polyestt1} \left\| \left[ \tau_t(W(f)), W(g) \right]  \right\|  &\leq  \min \left[2,  \frac{C_1 \| f \|_1 \| g \|_1}{|t|} \right].
\end{align}
\end{enumerate}
\end{thm}
\begin{proof}
We start by noticing that equation \eqref{eq:weylrel} together with the fact that all the Weyl operators are unitary imply
\begin{equation} \label{eq:easylrb}
\left\| \left[ \tau_t(W(f)), W(g) \right]  \right\| = \left| 1 - e^{i {\rm Im}[\langle T_tf, g \rangle]} \right| \leq \left| \langle T_tf, g \rangle \right|   \, ,
\end{equation}
for all $f,g \in \ell^2( \mathbb{Z}^2)$. On the other hand, Corollary~\ref{thm:unif} yields that there exists $C>0$ such that
\begin{equation}
\left| \langle T_tf, g \rangle \right|  \leq \|f\|_1 \|g\|_1  \frac {C} {|t|^{3/4}}, \quad \text{ for all } |t| \ge 1,
\end{equation}
proving \eqref{eq:polyestt34}.

The assumption that $x - y \notin B_{t\delta}(tV_3)$ for all $x \in X$ and $y \in Y$, guarantees that each term in $\langle T_tf, g \rangle$ can be estimated by one of the following:  \eqref{eq:poly56}, \eqref{eq:poly1}, or by the results of Theorem~\ref{thm:exp2}, proving \eqref{eq:polyestt56}. 
A similar observation also proves \eqref{eq:polyestt1}.
\end{proof}

For pairs of Weyl observables that are supported ``far" from each other, Theorems~\ref{thm:exp1} and~\ref{thm:exp2} provide Lieb-Robinson exponential bounds for the corresponding commutator norm.

\begin{thm}
Let $\tau_t$ be the harmonic dynamics on $\mathcal{W}( \ell^2( \mathbb{Z}^2))$ and sets $V_i$, $i=0, 1$, be defined in \eqref{eq:setsV}. For an arbitrary fixed $\delta>0$, assume that $X - Y \in tV_0$ with $\mathrm{dist}(X-Y, t V_1) \ge \delta$. Then there exist constants $C = C(\delta)$ and $\mu = \mu(\delta)$ such that
\begin{equation} \label{eq:LRboundV1}
\left\| \left[ \tau_t(W(f)), W(g) \right]  \right\|  \leq C \| f \|_1 \| g \|_1 e^{-\mu\, \mathrm{dist}(X-Y,\, tV_1)}.
\end{equation}
Moreover, for every $\mu > 0$ there exist constants 
$0 < v_\mu \leq \frac{1}{\mu}(1 + 2\sqrt{\lambda_1 + \lambda_2}\,\sinh(\mu/2))$ and
$C_\mu < \omega + 2\sqrt{\lambda_1 + \lambda_2}\cosh(\mu/2)$ such that
\begin{equation}  \label{eq:LRbound}
\left\| \left[ \tau_t(W(f)), W(g) \right]  \right\| \leq C_\mu e^{-\mu(\dist(X, Y)-v_\mu|t|)}.
\end{equation}
\end{thm}

Initially proven by Lieb and Robinson for non-relativistic quantum spin systems (\cite{lieb1972}), Lieb-Robinson bounds have been extended to various lattice systems (in particular, see \cite{nachtergaele2009} for infinite and \cite{raz2009} for finite harmonic and anharmonic lattices). In general, for non-relativistic models Lieb-Robinson bounds indicate the exponential decay with order $\mu$ of interactions outside of a light cone with velocity $v_\mu$. These estimates can also be used for establishing local existence and uniqueness of evolution for some infinite-dimensional models. 

Equation \eqref{eq:LRbound} improves the upper bound on $v_\mu$ (specifically, its behavior with respect to the on-site energy parameter $\omega$) as compared to the results in \cite{nachtergaele2009}. Equation \eqref{eq:LRboundV1} is a new form of a Lieb-Robinson bound, where velocity depends on the direction of movement and is determined by the shape of the outer boundary of the set $V_1$.

\section{Properties of the phase function}																					\label{sec:phasefunc}
Throughout this section we will be using the variables $(a, b)$ defined by
\begin{equation}																									\label{eq:ab}
a(k) = \cos k_1, \quad b(k)=\cos k_2,
\end{equation}
along with $k = (k_1, k_2)$. The function $\gamma$, defined in \eqref{gammaABk}, as a function of $(a, b)$ takes the form
\begin{equation}																								\label{gammaAB}
\gamma(a, b) = \left(\om^2 + 2\la_1(1-a) + 2\la_2(1-b)\right)^{1/2}.
\end{equation}
A straightforward calculation shows that the Hessian matrix of $\gamma$ can be written in the form
\begin{equation}																				\label{eq:HessianMatrix}
D^2\gamma(k) = \frac 1 {\gamma^3(k)} 
\left(
\begin{array}{c c}
\la_1 a \gamma^2(k) - \la_1^2 (1-a^2) & -\la_1 \la_2 \sin k_1 \sin k_2\\  -\la_1 \la_2 \sin k_1 \sin k_2 & \la_2 b \gamma^2(k) - \la_2^2 (1-b^2)
\end{array}
\right)
\end{equation}
and, thus,
\begin{align}																									
\notag \det D^2\gamma(k) &= \frac {\la_1 \la_2}{\gamma^4(k)}\left(ab\gamma^2(k) - \la_1 b(1-a^2) -\la_2 a(1-b^2) \right)\\ \label{eq:Hessianfunction} &= \frac {\la_1 \la_2}{\gamma^4(k)}\left(ab \om^2 - \la_1 b(1-a)^2 - \la_2 a(1-b)^2 \right).
\end{align}

\begin{prop} \label{prop:rank1}
For every choice of $\omega$, $\lambda_1, \lambda_2 > 0$, there is no $k \in [-\pi, \pi]^2$ such that 
$D^2\gamma(k)$ is the zero matrix.
\end{prop}
\begin{proof}
The off-diagonal entries vanish only if $\sin k_1 = 0$ or $\sin k_2 = 0$.  Without loss of generality, suppose
$\sin k_1=0$.  Then $a = \cos k_1 = \pm 1$, so that $\partial_{k_1}^2 \gamma(k)
 = \pm \lambda_1\gamma^{-1}(k) \not= 0$.  If $\sin k_2 = 0$, then $\partial_{k_2}^2 \gamma(k)$
is nonzero for similar reasons.
\end{proof}

One of our goals is to describe the set of zeros of the Hessian determinant,
\begin{equation}																							\label{eq:Phi1}
\Phi_1 = \{ k \in [-\pi, \pi]^2 : \det D^2\gamma(k) = 0\}.
\end{equation}
However, it will be convenient to first study the zeros of $\det D^2\gamma$ as a function of $(a, b)$:\begin{equation}																							\label{eq:Gamma11}
\Gamma_1 = \{ (a, b) \in [-1, 1]^2 : \det D^2\gamma(a, b) = 0\}.
\end{equation}
Using the notation
\begin{align}																						
\notag F(a, b) &= ab\gamma^2(a, b) - \la_1 b(1-a^2) -\la_2 a(1-b^2)\\
\label{funcF} &= ab \om^2 - \la_1 b(1-a)^2 - \la_2 a(1-b)^2,
\end{align}
we have that $\det D^2\gamma(a, b) = 0$ if and only if $F(a, b) = 0$ (since $\la_1, \la_2, \gamma(k) \ne 0$), and therefore, 
\begin{equation}																					\label{eq:Gamma1}
\Gamma_1 = \{ (a, b) \in [-1, 1]^2 : F(a, b) = 0\}.
\end{equation}
Sometimes it will be convenient to treat $F$ as a function of $k$, and in those cases we will keep the same notation, $F = F(k)$. Note also that with the notation \eqref{funcF} the Hessian matrix takes the form:
\begin{equation}																				\label{eq:HessianMatrix1}
D^2\gamma(k) = \frac 1 {\gamma^3(k)} 
\left(
\begin{array}{c c}
\la_1 \partial_b F & -\la_1 \la_2 \sin k_1 \sin k_2\\  -\la_1 \la_2 \sin k_1 \sin k_2 & \la_2 \partial_a F
\end{array}
\right).
\end{equation}

\begin{lem}																															\label{lem:Fprop}
The equation $F(a, b) = 0$ defines an implicit function $b=B_F(a)$ in $[-1, 1]^2$ that is locally continuously differentiable at any $(a, b) \in \Gamma_1$, $|b| \ne 1$. It has the following properties:

\begin{enumerate}
\item For any $a \in [-1, 1]$, there exists at most one $b \in [-1, 1]$ so that $F(a, b) = 0$.

\item For any $(a, b) \in \Gamma_1\backslash\{(0, 0)\}$, $|b| \ne 1,$ 
\begin{equation}																						\label{eq:Fder1}
\frac {dB_F}{da}\,(a, b) = - \frac {\la_1}{\la_2} \frac {b^2(1-a^2)}{a^2(1-b^2)} \le 0,
\end{equation}
and
\begin{equation}																					  \label{eq:Fder1at0}
\frac {dB_F}{da}\,(0, 0) = - \frac {\la_2}{\la_1}.
\end{equation}

\item The set $\Gamma_1$ defined in \eqref{eq:Gamma1} is the graph of $b = B_F(a)$, which consists of two continuous arcs $\Gamma_1^1 \cup \Gamma_1^2$ as displayed in Figure~\ref{Gamma1fig}. 
The first arc, $\Gamma_1^1,$ is located in the first quadrant and is convex. The second arc, $\Gamma_1^2$, passes through the second and fourth quadrant and is concave.
\end{enumerate}

The equation $F(a,b) = 0$ also defines an implicit function $a = A_F(b)$.  All of the statements in this lemma remain true if
$(a, A_F, \lambda_1)$ and $(b, B_F, \lambda_2)$ switch roles.
 
\begin{figure}[ht]
\begin{picture}(160,160)
\put(0,70){\vector(1,0){160}}											
\put(165,70){$a$}
\put(80,0){\vector(0,1){150}}
\put(70,147){$b$}
\thinlines
\qbezier(30,120)(30,120)(130,120)									
\qbezier(30,20)(30,20)(130,20)
\qbezier(30,120)(30,120)(30,20)
\qbezier(130,120)(130,120)(130,20)
\put(73,124){$1$}
\put(134,75){$1$}
\put(14,75){$-1$}
\put(65,10){$-1$}
\thicklines
\qbezier(30,90)(92,92)(95,20)											
\put(60, 90){$\Gamma_1^2$}
\qbezier(110,120)(111,101)(130,100)
\put(110, 90){$\Gamma_1^1$}
\end{picture}
\caption{Set $\Gamma_1$}
\label{Gamma1fig}
\end{figure}
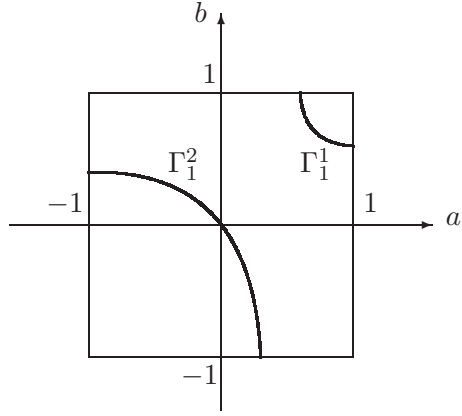
\end{lem}
\begin{proof}
For fixed $a \in [-1,1]$ the function $F(a,b)$ is quadratic with respect
to $b$ and has at most one solution in the interval $b\in[-1,1]$.  
The same is true if $b$ is held fixed and one seeks the value $a = A_F(b)$
for which $F(A_F(b),b) = 0$.  As a result $B_F$ is a well-defined function
over some subset of $[-1,1]$ and $A_F$ serves as its inverse.

One can write out the value of $B_F(a)$ explicitly using the quadratic formula
and derive all the stated properties from this expression.  We present a
more general approach here in preparation for subsequent computations
where an exact formula is not readily available.

The function $F$ is continuously differentiable and 
\begin{align*}
\frac {\partial F}{\partial b}(a,b) = a\om^2 - \la_1 (1-a)^2 + 2\la_2 a(1-b) &= \frac 1 b \left(ab\om^2 - \la_1 b(1-a)^2 \right) + 2\la_2 a(1-b)\\
&= \frac{1}{b}\big(F(a,b) + \lambda_2 a(1-b)^2\big) + 2\lambda_2 a(1-b).
\end{align*}
Therefore, 
\begin{equation}																								\label{eq:bderbF}
b\frac {\partial F}{\partial b}(a,b)
= \la_2 a(1-b^2) \quad \text{for all } (a, b) \in \Gamma_1,
\end{equation} 
and this quantity is nonzero so long as $a \ne 0$ and $|b| < 1$.  When $a =0$ one can compute directly that $\frac {\partial F}{\partial b}(0,b) = - \la_1 \ne 0$.  Therefore 
\begin{equation*}
\left. \frac {\partial F}{\partial b}(a,b)\right|_{\Gamma_1} = 0
\quad \text{if and only if } |b| = 1.
\end{equation*}
An identical argument applied to the variable $a$ shows that
\begin{equation}  \label{eq:aderaF}
a \frac{\partial F}{\partial a}(a,b) = \lambda_1 b(1-a^2) \quad \text{for all }(a,b) \in \Gamma_1
\end{equation}
and furthermore that 
$\left. \frac {\partial F}{\partial a}(a,b)\right|_{\Gamma_1} = 0$
if and only if $|a| = 1$.


In order to prove equation \eqref{eq:Fder1}, we differentiate $F(a, b) = 0$ implicitly with respect to $a$.
Taking advantage of~\eqref{eq:bderbF} and~\eqref{eq:aderaF} the resulting
expression reduces to
\begin{equation}
\frac{db}{da} = -\frac{b}{a}
\bigg( \frac{a\, \partial_aF(a,b)}{b\, \partial_bF(a,b)}\bigg)
 = -\frac{\lambda_1 b^2(1-a^2)}{\lambda_2 a^2(1-b^2)}
\end{equation}
for all $(a,b) \in \Gamma_1$ away from the origin. At the origin,~\eqref{eq:Fder1at0} 
is obtained directly from the facts that 
$\frac{\partial F}{\partial a}(a,0) = -\lambda_2$ and $\frac{\partial F}{\partial b}(0,b) = -\lambda_1$.
This is consistent with the implicit derivative in~\eqref{eq:Fder1} since both statements  demand that the ratio $b^2/a^2$ converges to $(\lambda_2/\lambda_1)^2$ as $(a,b)$ approaches the origin along $\Gamma_1$.

By definition $\Gamma_1$ must be the graph of $B_F$, which is continuously
differentiable with negative slope whenever it lies inside $(-1,1)^2$ and has
slope zero when $|a| = 1$.  
Given that $0 < B_F(-1) < B_F(1) < 1$ and $0 < A_F(-1) < A_F(1) < 1$, 
it follows that $\Gamma_1$ consists of two separate arcs.  One arc, denoted by
$\Gamma_1^1$, connects the points $(A_F(1),1)$ and $(1, B_F(1))$
within the first quadrant.  The second arc, denoted by $\Gamma_1^2$,
connects $(-1, B_F(-1))$ to $(A_F(-1),-1)$ and
passes through the origin (since $F(0,0) = 0$) along the way.


Finally, a routine derivation shows that
\begin{equation} \label{eq:Fder2}
\left.\frac{d^2b}{da^2}\right|_{\Gamma_1} = \frac{2\la_1}{\la_2^2}\frac{b^2}{a^4(1-b^2)^3}\left(\la_1 b(1-a^2)^2 + \la_2 a (1-b^2)^2 \right),
\end{equation}
which is clearly positive if $a, b > 0$, thus proving that $\Gamma_1^1$ is convex. Using again that $F(a, b) = 0$, we rewrite the second derivative in the form
\begin{align*}
\frac{d^2b}{da^2} &= \frac{2\la_1}{\la_2^2}\frac{b^2}{a^4(1-b^2)^3}\left(F(a,b) + \lambda_1 b(1-a^2)^2
+ \lambda_2 a(1-b^2)^2\right) \\
&= \frac{2\la_1}{\la_2^2}\frac{b^2}{a^4(1-b^2)^3}\,ab\left(\omega^2 + \lambda_1(1-a)^2(2+a)
+ \lambda_2 (1-b)^2(2+b)\right)
\end{align*} 
So long as $0 < |a|, |b| < 1$, the sign of this second derivative is determined
by the sign of $ab$, which is negative everywhere on $\Gamma_1^2$ except the origin.
A separate calculation shows that
\begin{equation}
\frac{d^2b}{da^2}(0, 0) = \frac{-2\la_2(\om^2 + 2\la_1+ 2\la_2)}{\la_1^2} < 0,
\end{equation}
finishing the proof.  The above expression can be simplified further by noting that $\omega^2 + 2\la_1 + 2\la_2 = \gamma^2$
when $a = b = 0$.
\end{proof}

\begin{rem} \label{rem:Phi}
Recalling that $a = \cos k_1$ and $b = \cos k_2$ for points $(k_1, k_2) \in [-\pi, \pi]^2$, we can reconstruct the graph of $\Phi_1$ (defined in \eqref{eq:Phi1}) from the graph of $\Gamma_1$ (see Figure \ref{Phifig}). The arc $\Gamma_1^1$ corresponds to the closed curve around zero. The origin in the $ab$-plane has the four points $(\pm \pi/2, \pm \pi/2)$ as its inverse image and the arc $\Gamma_1^2$ turns into the closed curve around the point $(\pi, \pi)$ on the compactified torus $[-\pi, \pi]^2$.

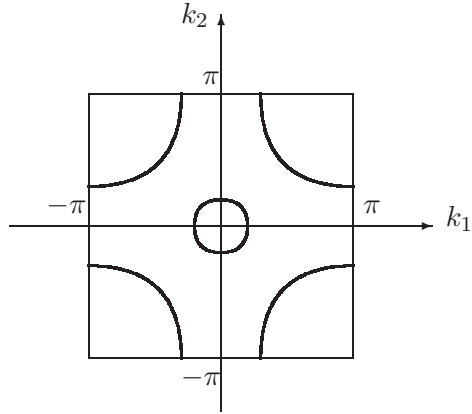
\begin{figure}[ht]
\begin{picture}(160,160)
\put(0,70){\vector(1,0){160}}											
\put(165,70){$k_1$}
\put(80,0){\vector(0,1){150}}
\put(65,147){$k_2$}
\thinlines
\qbezier(30,120)(30,120)(130,120)									
\qbezier(30,20)(30,20)(130,20)
\qbezier(30,120)(30,120)(30,20)
\qbezier(130,120)(130,120)(130,20)
\put(73,124){$\pi$}
\put(134,75){$\pi$}
\put(14,75){$-\pi$}
\put(65,10){$-\pi$}
\thicklines
\qbezier(30,55)(65,55)(65,20)											
\qbezier(30,85)(65,85)(65,120)
\qbezier(95,20)(95,55)(130,55)
\qbezier(95,120)(95,85)(130,85)
\qbezier(70,70)(70,60)(80,60)											
\qbezier(90,70)(90,60)(80,60)	
\qbezier(90,70)(90,80)(80,80)	
\qbezier(80,80)(70,80)(70,70)	
\end{picture}
\caption{Set $\Phi_1$}
\label{Phifig}
\end{figure}
\end{rem}
By definition a function $\phi_v(k) = k \cdot v -\gamma(k)$ may have a degenerate critical point only at $k^* \in \Phi_1$, and this occurs with the choice $v = \nabla\gamma(k^*)$. 
For each $k \in \Phi_1$, let $\xi = \xi(k)$ be an eigenvector of $D^2\gamma(k)$ corresponding to its zero eigenvalue (this is unique up to scalar multiplication by Proposition~\ref{prop:rank1}).  It follows that
$\partial_{\xi}^2 \gamma (k) = 0$, where the notation $\partial_{\xi} = \xi \cdot \nabla$ indicates a directional derivative as in Section~\ref{sec:proofs}. 
According to the partition~\eqref{eq:setK}, points $k \in \Phi_1$ belong to either $K_2$ or $K_3$ depending on whether the third derivative of $\gamma$ in the direction of $\xi$ is also zero.  The following result is helpful. 
\begin{lem}																																	\label{lm:3der}
Let $U \subset \R^d$ be a neighborhood of a point $k_0$ and let $f \in C^3(U)$. Assume that the Hessian matrix $D^2 f$ has a zero eigenvalue of multiplicity one at $k_0$ and let $\xi$ be a corresponding eigenvector.  Then 
\begin{equation}
\partial_{\xi}^3 f(k_0) = 0
\end{equation}
 if and only if 
\begin{equation}																														\label{eq:3dereq}
\partial_{\xi} (\det D^2 f)(k_0) = 0.
\end{equation}
\end{lem}
\begin{proof}
Apply a unitary change of variables to change the coordinate system to one that diagonalizes the matrix $(D^2 f)(k_0)$ and in which $\xi$ points in the direction of $e_1$. In the new system, the only non-zero term of the gradient of $\det(D^2 f)(k)$ at $k_0$ is the gradient of $(\partial^2_{11}f)(k)$ at $k_0$ multiplied by a nonzero scalar - the product of all nonzero eigenvalues of $(D^2 f)(k_0)$. On the other hand, we have that 
\begin{equation}
(\partial^2_{11} f)(k) = \frac 1 {\|\xi\|^2} \xi^T D^2f(k) \xi = \frac 1 {\|\xi\|^2} \partial_{\xi}^2 f (k),
\end{equation}
showing that $(\nabla (\det D^2 f))(k_0)$ is a non-zero scalar multiple of $\nabla(\partial_{\xi}^2 f)(k_0)$.
\end{proof}

Lemma~\ref{lm:3der} allows us to determine whether points $k \in \Phi_1$ satisfy $\partial_{\xi}^3 \gamma(k) = 0$ by identifying the set of solutions of
\begin{equation}			\label{eq:D2xider}
\partial_{\xi} (\det D^2 \gamma)(k) =0.
\end{equation}
Using equations \eqref{eq:Hessianfunction}, \eqref{funcF}, and notation \eqref{eq:ab} we have
\begin{equation}																											\label{eq:gradH}
\nabla \det D^2\gamma(k)\big|_{k \in \Phi_1} = \frac {\la_1 \la_2}{\gamma^4(k)}\nabla F(k) = -\frac {\la_1 \la_2}{\gamma^4(k)}\left(\partial_a F\sin k_1,\;  \partial_b F\sin k_2\right). \end{equation}
The components of $\xi$ can be constructed from the elements of the matrix \eqref{eq:HessianMatrix1}, with one
possible choice being
\begin{equation}	\label{eq:xi}																		
\xi(k) = \left(
\begin{array}{c}
\partial_a F \\ \la_1 \sin k_1 \sin k_2
\end{array}
\right).
\end{equation}
Note that $\xi(k)$ vanishes as $\sin k_1$ approaches zero along $\Phi_1$ (for example by applying~\eqref{eq:aderaF}), and
in fact
\begin{equation}
\frac{1}{\sin k_1} \xi(k) \to \begin{pmatrix} 0 \\ \lambda_1 \sin k_2 \end{pmatrix} \not= 0
\end{equation}
when $k_1 \to 0,\,\pm \pi$ along this curve.  The combination of~\eqref{eq:xi} with~\eqref{eq:gradH} yields
\begin{equation}																			\label{eq:detHder}
\partial_{\xi} (\det D^2 \gamma)(k) = -\frac {\la_1 \la_2 \sin k_1}{\gamma^4(k)}\left( (\partial_a F)^2 + \la_1(1 - b^2)\partial_b F\right), \quad k \in \Phi_1.
\end{equation}
One should not be concerned with the vanishing of $\sin k_1$ in this formula as it can be counteracted by modifying~\eqref{eq:xi} by a suitable scalar multiple.  Vanishing of the second factor determines whether $k\in \Phi_1$ belongs to $K_2$ or $K_3$. Using \eqref{eq:bderbF} and~\eqref{eq:aderaF}, we have
\begin{align*}
a^2b \left( (\partial_a F)^2 + \la_1(1 - b^2)\partial_b F \right)\Big|_{\Gamma_1} &= \la_1^2 b^3(1-a^2)^2 + \la_1\la_2 a^3 (1-b^2)^2
\end{align*}
and thus if $k \in \Phi_1$ (equivalently if $(a,b)\in \Gamma_1$), then \eqref{eq:D2xider} holds only if
\begin{equation}																										\label{eq:Gtilde}
\tilde{G}(a, b) = \la_1 b^3(1-a^2)^2 + \la_2 a^3 (1-b^2)^2 = 0.
\end{equation}
The function $\tilde{G}$ is symmetric in $a$ and $b$ and is rather elegant, but it turns out not to be ideal for our purposes.  Restricting our view to $k \in \Phi_1$, we are again free to add any multiple of $F$ to $\tilde{G}$ and work with that object instead. We therefore introduce 
\begin{align}																										\label{eq:GtildeGF}
G(a, b) &= \tilde{G}(a, b) + a^2 b^2 F(a, b)\\
\notag &=\om^2 a^3 b^3 + \la_1 b^3(1-3a^2+2a^3) + \la_2 a^3 (1-3b^2+2b^3).
\end{align} 
The function $G$ maintains the property that among all $k \in \Phi_1$, \eqref{eq:D2xider} holds only if $G(a, b) = 0$. We will describe some features of the set
\begin{equation}																											\label{eq:Gamma2}
\Gamma_2 = \{ (a, b) \in [-1, 1]^2 : G(a, b) = 0\}
\end{equation}																	
as an independent object before seeking out its intersection with $\Gamma_1$. The following result is an analog of Lemma~\ref{lem:Fprop}  and we provide it for the sake of completeness.
\begin{lem}																													\label{lem:Gprop}
The set $\Gamma_2$ defined in \eqref{eq:Gamma2} is nonempty and is of the form displayed in Figure~\ref{Gamma2fig}: it consists of two continuous arcs, $\Gamma_2 = \Gamma_2^1 \cup \Gamma_2^2$.
The equation $G(a, b) = 0$ defines an implicit function in $[-1, 1]^2$ that is locally continuously differentiable with respect to $a$ at any $(a, b) \in \Gamma_2$, $|b| \ne 1$. In particular, the arc $\Gamma_2^2$ represents the graph of a function which we denote by $b=B_G(a)$. The following properties hold:

(1) For any $(a, b) \in \Gamma_2^2\backslash\{(0, 0)\}$, $|b| \ne 1,$ 
\begin{equation}																						\label{eq:Gder1}
\frac {dB_G}{da}\,(a, b) = - \frac {\la_1}{\la_2} \frac {b^4(1-a^2)}{a^4(1-b^2)} \le 0,
\end{equation}
and
\begin{equation}																					  \label{eq:Gder1at0}
\frac {dB_G}{da}\,(0, 0) = - \left(\frac {\la_2}{\la_1}\right)^{1/3}.
\end{equation}

(2) The first arc, $\Gamma_2^1,$ passes through the third quadrant. The second arc, $\Gamma_2^2$, is located in the second and fourth quadrant. 

The arc $\Gamma_2^2$ is also the graph of a function $a = A_G(b)$. All of the statements (or their analogs) in this lemma remain true if $(a, A_G, \lambda_1)$ and $(b, B_G, \lambda_2)$ switch roles.
 
\begin{figure}[ht]
\begin{picture}(160,160)
\put(0,70){\vector(1,0){160}}											
\put(165,70){$a$}
\put(80,0){\vector(0,1){150}}
\put(70,147){$b$}
\thinlines
\qbezier(30,120)(30,120)(130,120)									
\qbezier(30,20)(30,20)(130,20)
\qbezier(30,120)(30,120)(30,20)
\qbezier(130,120)(130,120)(130,20)
\put(73,124){$1$}
\put(134,75){$1$}
\put(14,75){$-1$}
\put(65,10){$-1$}
\thicklines
\qbezier(60,120)(70,49)(130,50)											
\put(52, 90){$\Gamma_2^2$}
\qbezier(30,50)(60,50)(65,20)
\put(35, 55){$\Gamma_2^1$}
\end{picture}
\caption{Set $\Gamma_2$}
\label{Gamma2fig}
\end{figure}
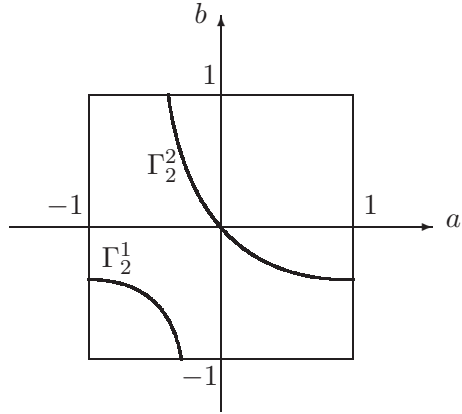
\end{lem}
\begin{proof}
The proof is largely identical to that of Lemma~\ref{lem:Fprop} and we omit the common details.
One difference is that a slope at the origin cannot be determined from the ratio of $\frac{\partial G}{\partial a}(0,0)$ and
$\frac{\partial G}{\partial b}(0,0)$, as both quantities are zero already.  

Note that $G(0,b) = \lambda_1b^3$ is zero only if $b=0$.  When $a \not= 0$, write $r(a) = \frac{b}{a}$ to obtain the
expression
\begin{equation}
G(a,r) = a^3\big( (\lambda_1 r^3 + \lambda_2) + a^2(-3\lambda_1 r^3-3\lambda_2 r^2)
+ a^3(\omega^2 +2\lambda_1+ 2\lambda_2)r^3 \big).
\end{equation}
For a fixed value of $a$, the solutions of $G(a,r) = 0$ occur at the roots of a cubic polynomial whose coefficients
depend smoothly on $a$.  When $a=0$ the polynomial is $\lambda_1r^3 + \lambda_2$, which has a single transversal
root at $r = -(\frac{\lambda_2}{\lambda_1})^{1/3}$.  The implicit function theorem provides a neighborhood of
$a=0$ and a continuous function $r(a)$ along which $G(a,r(a)) = 0$.  

By definition $\dfrac{dB_G}{da}(0,0)
= \lim\limits_{a\to 0} r(a) = -(\frac{\lambda_2}{\lambda_1})^{1/3}$.  Once again the result is consistent with the
general implicit derivative~\eqref{eq:Gder1} because the ratio $(b^4/a^4)$ converges to $(\lambda_2/\lambda_1)^{4/3}$
as $(a,b)$ approaches the origin along $\Gamma_2$.

\end{proof}
We are most interested in the intersection points of $\Gamma_1$ and $\Gamma_2$, which are described in the following result.
\begin{lem}																													\label{lm:Ga1Ga2int}
Let the curves $\Gamma_1$ and $\Gamma_2$ be defined by \eqref{eq:Gamma1} and \eqref{eq:Gamma2}, respectively. Then 
\begin{enumerate}
\item[(i)] if $\la_1 < \la_2$, $\Gamma_1 \cap \Gamma_2 = \{(0, 0), (a^*, b^*)\}$, with $a^* < 0 < b^*$,
\item[(ii)] if $\la_1 > \la_2$, $\Gamma_1 \cap \Gamma_2 = \{(0, 0), (a^*, b^*)\}$, with $b^* < 0 < a^*$,
\item[(iii)] if $\la_1 = \la_2$, $\Gamma_1 \cap \Gamma_2 = \{(0, 0)\}$.
\end{enumerate}
In the last case, we will use the notation $(a^*, b^*) = (0, 0)$.
\end{lem}
\begin{proof}
Consider the case $\la_1 < \la_2$. The proof in the case $\la_1 > \la_2$ is identical. 

Note that since the origin belongs to both $\Gamma_1$ and $\Gamma_2$, it is obviously in their intersection. Lemma~\ref{lem:Fprop} and Lemma~\ref{lem:Gprop} show that $\Gamma_2^1 \cap \Gamma_1 = \Gamma_1^1 \cap \Gamma_2 = \emptyset$ and thus, $\Gamma_1\cap\Gamma_2 = \Gamma_2^2 \cap \Gamma_1^2$. Next, according to Lemma~\ref{lem:Fprop}, $\Gamma_1^2$ is concave, and with the assumption $\la_2 > \la_1$  its slope at zero is less than $-1$ (see equation \eqref{eq:Fder1at0}). As a result, for all $(a, b) \in \Gamma_1^2$ in the fourth quadrant, $\abs{b} > \abs{a}$, except for the origin. Define 
\begin{equation}																			\label{eq:aplus}
a^+ = \max \{a: (a, b) \in \Gamma_1 \cap \Gamma_2\}.
\end{equation}
If we assume that $\Gamma_1^2$ intersects $\Gamma_2^2$ in the fourth quadrant away from the origin, we have that $0 < a^+ < \tilde{a}$, where $\tilde{a} = A_F(-1)$. Comparing formulas \eqref{eq:Fder1} and \eqref{eq:Gder1} gives
\begin{equation}																								\label{eq:FGderivatives}
\frac {dB_G}{da}\,(a) = \frac {dB_F}{da}\,(a) \frac {b^2}{a^2}, \quad \text{ for all } (a, b) \in \Gamma_1 \cap \Gamma_2, \; (a, b) \ne (0, 0).
\end{equation}
Both derivatives are negative, and furthermore $|b| >|a|$ along this part of the curve $\Gamma_1^2$.  Consequently
\begin{equation}
\frac {dB_G}{da}\,(a^+) < \frac {dB_F}{da}\,(a^+),
\end{equation}
and thus, $B_G(a) < B_F(a)$ for all $a \in (a^+, a^+ + \epsilon)$ for some small $\epsilon > 0.$ On the other hand, it follows from the results of Lemma~\ref{lem:Gprop} that $B_G(\tilde{a}) > -1 = B_F(\tilde{a})$, implying that there must exist another point $a' \in [a^+ + \epsilon, \tilde{a})$ such that $B_G(a') = B_F(a')$. This contradicts the definition of $a^+$ and we may conclude that $\Gamma_1 \cap \Gamma_2 = \{(0, 0)\}$ in the fourth quadrant.  

We claim that $\Gamma_1^2$ intersects $\Gamma_2^2$ exactly once in the second quadrant away from the origin. First, since 
\begin{equation}
\frac {dB_F}{da}\,(0) < \frac {dB_G}{da}\,(0),
\end{equation}
there is a small $\epsilon > 0$ so that $B_F(a) > B_G(a)$ for all $a \in (-\epsilon, 0)$.  However for the value $\hat{a} = A_G(1)$, one has $B_F(\hat{a}) < 1 = B_G(\hat{a})$. The Intermediate Value Theorem implies the existence of $a' \in (\hat{a}, 0)$ such that $B_F(a') = B_G(a')$, giving rise to at least one non-origin point of intersection of $\Gamma_1^2$ and $\Gamma_2^2$ in the second quadrant.

Note that $G(a,-a) = -(\omega^2 + 4\lambda_1)a^6 + (\lambda_2-\lambda_1)a^3(1 - 3a^2 -2a^3) < 0$ over the interval $a \in [-1,0)$.  Together with~\eqref{eq:Gder1at0} that implies that $\Gamma_2^2$ lies above the line $b =-a$ within the second quadrant, and it follows from~\eqref{eq:FGderivatives} that $0 > \frac{dB_F}{da}(a) > \frac{dB_G}{da}(a)$ whenever $\Gamma_1^2$ and $\Gamma_2^2$ intersect with $a < 0$.  On the other hand, if there were multiple points of intersection, the orientation of crossing would have alternating signs.  One concludes that $\Gamma_1^2 \cap \Gamma_2^2$ contains a single point $(a^*, b^*)$ in the second quadrant along with the origin.


Now consider the case $\la_1 = \la_2 = \la$, and let $(a_0, b_0) \in \Gamma_1 \cap \Gamma_2$. Then $\tilde{G}(a_0, b_0)=0$, where $\tilde{G}$ is defined in \eqref{eq:Gtilde}. Since $\lambda_1$ and $\lambda_2$ are equal, $\tilde{G}$ admits the factorization 
\begin{equation}
\tilde{G}(a, b) = \la(a + b)(a^2 + b^2 + ab(a^2b^2 - 2ab - 1)).
\end{equation}
The second factor is zero if and only if $|a| = |b| = 1$, however those points do not belong to $\Gamma_1$. Hence, $b_0 = -a_0$. Plugging this into~\eqref{funcF}, one can see that when $\la_1 = \la_2$, $F(a_0, -a_0) = 0$ if and only if $a_0 = 0$.  The intersection of $\Gamma_1$ and $\Gamma_2$ at the origin is not transversal in this case, but instead the two curves are tangent without crossing. 
\end{proof}

\begin{rem}																														\label{rm:Kstar}
Introduce the set $K^*$ as the pre-image of the point $(a^*, b^*)$ under the map \eqref{eq:ab}. The set $K^*$ consists of four points on the set $\Phi_1$ that are located as shown on Figure~\ref{PhiKfig}. Note that in the case $\la_1 = \la_2$, $K^* = \left\{(\pm \frac \pi 2, \pm \frac \pi 2)\right\}$. 

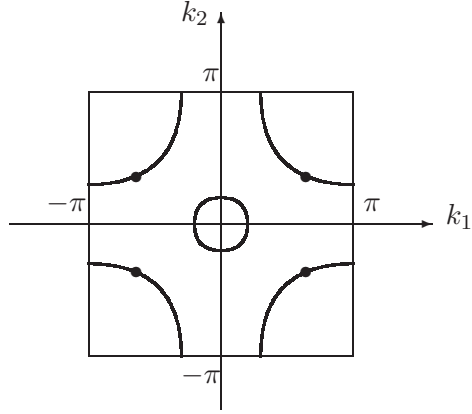
\begin{figure}[ht]
\begin{picture}(160,160)
\put(0,70){\vector(1,0){160}}											
\put(165,70){$k_1$}
\put(80,0){\vector(0,1){150}}
\put(65,147){$k_2$}
\thinlines
\qbezier(30,120)(30,120)(130,120)									
\qbezier(30,20)(30,20)(130,20)
\qbezier(30,120)(30,120)(30,20)
\qbezier(130,120)(130,120)(130,20)
\put(73,124){$\pi$}
\put(134,75){$\pi$}
\put(14,75){$-\pi$}
\put(65,10){$-\pi$}
\thicklines
\qbezier(30,55)(65,55)(65,20)											
\qbezier(30,85)(65,85)(65,120)
\qbezier(95,20)(95,55)(130,55)
\qbezier(95,120)(95,85)(130,85)
\qbezier(70,70)(70,60)(80,60)											
\qbezier(90,70)(90,60)(80,60)	
\qbezier(90,70)(90,80)(80,80)	
\qbezier(80,80)(70,80)(70,70)	
\put(112,88){\circle*{4}}
\put(112,52){\circle*{4}}
\put(48,88){\circle*{4}}
\put(48,52){\circle*{4}}
\end{picture}
\caption{Set $\Phi_1$ along with the four-point set $K^*$}
\label{PhiKfig}
\end{figure} 
\end{rem}

The following is an easy consequence of Lemma~\ref{lm:Ga1Ga2int}. 

\begin{cor}																						\label{cor:Kstar}
Assume that $\det D^2\gamma(k) = 0$ and let $\xi = \xi(k)$ be an eigenvector of $D^2\gamma(k)$ corresponding to the zero eigenvalue. Let $K^*$ be the set defined in Remark \ref{rm:Kstar}. Then $\partial_{\xi}^3 \gamma (k) = 0$ if and only if $k \in K^*$.
\end{cor}
\begin{proof}
First recall that the assumption $\det D^2\gamma(k) = 0$ is equivalent to the fact that the corresponding $(a, b) \in \Gamma_1$ (see \eqref{eq:Gamma11}). 

Assume that $\partial_{\xi}^3 \gamma (k) = 0$. In Lemma~\ref{lm:3der} we showed that under our main assumption this condition is equivalent to  \eqref{eq:D2xider}. This, in turn, implies that $G(a, b) = 0$ (or $(a, b) \in \Gamma_2$), where $(a, b)$ is an image of $k$ under the map \eqref{eq:ab}. We may therefore conclude that $(a, b) \in \Gamma_1 \cap \Gamma_2$. According to Lemma~\ref{lm:Ga1Ga2int}, in the case $\la_1 = \la_2$, $(a, b) = (a^*, b^*)$ and thus, $k \in K^*$. If $\la_1 \ne \la_2$, $(a, b)$ is either $(a^*, b^*)$ or the origin. However, a direct inspection (see equation \eqref{eq:detHder}) shows that when $(a, b) = (0, 0)$,
\begin{equation}																							\label{eq:1deratzero}
\partial_{\xi} (\det D^2 \gamma)(k) = \pm \frac {\la_1 \la_2}{\gamma^4(k)}\left( \la_1^2 - \la_2^2 \right) \ne 0,
\end{equation}
and we can exclude the origin from consideration. Thus, $(a, b) = (a^*, b^*)$ and $k \in K^*$.

The proof in the reverse direction is similar. Assume $k \in K^*$, then it corresponds to $(a^*, b^*)$. Again, in the case $\la_1 = \la_2$, $(a^*, b^*) = (0, 0)$, and by \eqref{eq:1deratzero}, $\partial_{\xi} (\det D^2 \gamma)(k)  = \partial_{\xi}^3 \gamma (k) = 0$. If $\la_1 \ne \la_2$, both $a^*$ and $b^*$ are different from zero. Comparing equations \eqref{eq:detHder} -- \eqref{eq:GtildeGF}, we see that $G(a^*, b^*)$ is a nonzero multiple of $a^*(b^*)^2 \partial_{\xi} (\det D^2 \gamma)(k)$. Since $G(a^*, b^*) = 0$ and $a^*(b^*)^2 \ne 0$, we conclude $\partial_{\xi} (\det D^2 \gamma)(k)  = \partial_{\xi}^3 \gamma (k) = 0$.

\end{proof}

\begin{lem}																														\label{lm:fourthder}
Let $\gamma$ be defined by \eqref{gammaAB} and let $k^* \in K^*$, where $K^*$ is defined in Remark~\ref{rm:Kstar}. Furthermore, let $\xi = \xi(k^*)$ be the eigenvector of the Hessian matrix of $\gamma$ at $k^*$ corresponding to the zero eigenvalue and $\xi^{\perp}$ be a vector orthogonal to $\xi$ and of the same magnitude. Then 
\begin{equation}																												\label{eq:fourthder}
\left(\partial^4_{\xi}\gamma \partial^2_{\xi^\perp} \gamma - 3 \left(\partial^2_{\xi} \partial_{\xi^\perp}\gamma\right)^2 \right)(k^*) \neq 0.
\end{equation}
\end{lem}
\begin{proof}
To prove inequality \eqref{eq:fourthder}, it is enough to show that
\begin{equation}
\left(\partial^4_{\xi}\gamma \partial^2_{\xi^\perp} \gamma - 2 \left(\partial^2_{\xi} \partial_{\xi^\perp}\gamma\right)^2 \right)(k^*) < 0.
\end{equation} 
The above expression admits the following short representation that we will use in our calculations, 
\begin{equation}																						\label{eq:det2der}
\left(\partial^4_{\xi}\gamma \partial^2_{\xi^\perp} \gamma - 2 \left(\partial^2_{\xi} \partial_{\xi^\perp}\gamma\right)^2 \right)(k^*) = \|\xi\|^4\left(\partial^2_{\xi} \det D^2\gamma\right)(k^*).
\end{equation}
Let us first prove \eqref{eq:det2der}. Indeed, we can re-write $\det D^2\gamma$ in the new coordinates as
\begin{equation}
\|\xi\|^4\det D^2\gamma = \partial^2_{\xi} \gamma \partial^2_{\xi^\perp} \gamma - \left(\partial_{\xi} \partial_{\xi^\perp} \gamma\right)^2.
\end{equation}
Differentiating the above equation we obtain
\begin{equation*}
\|\xi\|^4 \partial^2_{\xi} \det D^2\gamma = \partial^4_{\xi} \gamma \partial^2_{\xi^\perp} \gamma + 2 \partial^3_{\xi} \gamma \partial_{\xi} \partial^2_{\xi^\perp} \gamma + \partial^2_{\xi} \gamma \partial^2_{\xi} \partial^2_{\xi^\perp} \gamma - 2 \left(\partial^2_{\xi} \partial_{\xi^\perp}\gamma\right)^2 - 2 \partial_{\xi} \partial_{\xi^\perp} \gamma \partial^3_{\xi} \partial_{\xi^\perp} \gamma.
\end{equation*}
At the point $k^*$, the quantities $\partial^2_{\xi} \gamma$, $\partial_{\xi} \partial_{\xi^\perp} \gamma$, and $\partial^3_{\xi} \gamma$ all vanish, thus the second, third and fifth term of the above equation vanish as well, proving \eqref{eq:det2der}. 

Next, it is easy to see that 
\begin{equation}
\left(\partial^2_{\xi} \det D^2\gamma\right)(k^*) = \left(\partial^2_{\xi} F\right)(k^*) \frac {\la_1 \la_2}{ \gamma^4(k^*)},
\end{equation}
where $F$ is defined in \eqref{funcF}. 
Finally, a direct calculation shows that in the case $\la_1 \ne \la_2$,
\begin{equation}
\left(\partial^2_{\xi} F \right)(k^*) = 2 \|\xi\|^2 \gamma^2(k^*) \frac{a^* b^* (1-(a^*)^2) (1-(b^*)^2)}{(a^*)^2 (1-(b^*)^2) + (b^*)^2 (1-(a^*)^2)} < 0,
\end{equation}
and in the case $\la_1 = \la_2 =\la$,
\begin{equation}
\left(\partial^2_{\xi} F \right)(k^*) = -\|\xi\|^2 (\omega^2 + 2(\lambda_1 + \lambda_2)) < 0,
\end{equation}
finishing the proof. 
\end{proof}

\subsection{Proof of Proposition~\ref{prop:Vpicture}}
The curves of $\Phi_1$ consist of points where $\det D^2\gamma(k) = 0$, which are also the points
where the ``velocity map" $\mathcal{V}(k) = \nabla\gamma(k)$ does not satisfy the hypotheses
of the inverse function theorem.  As a result the boundary of $\mathcal{V}([-\pi, \pi]^2)$
must be a subset of  $\Psi_1 \cup \Psi_2 = V(\Phi_1)$ as defined in Proposition~\ref{prop:Vpicture}.

Recall from Remark~\ref{rem:Phi} that $\Phi_1$ has one closed curve around zero
and a second closed curve around the point $(\pi, \pi)$.  Let $\Psi_1$ be the image
of the former under the velocity map and let $\Psi_2$ be the image of the latter.
The analysis of $\Psi_1$ is more straightforward because the
points of $K^*$ are not involved.

In vector form, the velocity map is 
\begin{equation*}
\nabla \gamma(k) =
\frac{1}{\gamma(k)}\begin{pmatrix}\lambda_1 \sin k_1 \\ \lambda_2 \sin k_2\end{pmatrix}.
\end{equation*}
Thus points $k$ in a given ``quadrant'' of the torus $[-\pi,\pi]^2$ are mapped to the same
quadrant of $\R^2$.

The tangent line to $\Phi_1$ always points in the direction normal to
$\nabla \det D^2\gamma(k)$, which by~\eqref{eq:gradH} is also normal to
$\nabla F(k)$.  Suppose $k$ travels along $\Phi_1$ with instantaneous velocity
$\begin{pmatrix} -\hskip-.12in&\partial_bF \sin k_2 \\
&\partial_a F \sin k_1 \end{pmatrix}$.  Then $k$ follows either loop of $\Phi_1$
through
the four quadrants of the compactified torus in order, and $\mathcal{V}(k)$
must wind once around the origin.

At a local level, the differential $d\mathcal{V}(k)$ is the Hessian
matrix $D^2\gamma(k)$, whose image is spanned by the direction $\xi^\perp$.
Plugging~\eqref{eq:HessianMatrix1} into the Leibniz rule we determine that
$\mathcal{V}(k)$ moves with velocity
\begin{equation}
\begin{aligned}
\frac 1 {\gamma^3(k)} 
&\begin{pmatrix}
\la_1 \partial_b F & -\la_1 \la_2 \sin k_1 \sin k_2\\  -\la_1 \la_2 \sin k_1 \sin k_2 & \la_2 \partial_a F
\end{pmatrix}
\begin{pmatrix}
-\partial_bF \sin k_2 \\ \partial_a F \sin k_1
\end{pmatrix} \\
& \hskip 1.5 in= 
\frac{1}{\gamma^3(k)}
\begin{pmatrix}
-\lambda_1(\partial_b F)^2 -\lambda_1\lambda_2 (1-a^2) \partial_a F \sin k_2 \\
\lambda_1\lambda_2 (1-b^2) \partial_b F \sin k_1 + \lambda_2(\partial_a F)^2 \sin k_2
\end{pmatrix} \\
&\hskip 1.5in = 
\frac{\lambda_1 \lambda_2}{\gamma^3(k)} 
\left[ \frac{\lambda_1 b^3(1-a^2)^2 + \lambda_2 a^3(1-b^2)^2}{a^2b} \right]
\begin{pmatrix}
-\frac{a}{b} \sin k_2 \\ \sin k_1
\end{pmatrix} \\
&\hskip 1.5in = 
\frac{\lambda_1 \lambda_2}{\gamma^3(k)} 
\left[ \frac{\tilde{G}(a,b)}{a^2b} \right]
\begin{pmatrix}
-\frac{a}{b} \sin k_2 \\ \sin k_1
\end{pmatrix}.
\end{aligned}
\end{equation}
Identities~\eqref{eq:bderbF} and~\eqref{eq:aderaF} are used multiple times
between the second and third lines.

The prefactor of $\lambda_1 \lambda_2 \gamma^{-3}(k)$ is positive for all $k$.
The factor of $\tilde{G}/(a^2b)$ is strictly positive as $k$ traces out the loop of
$\Phi_1$ circling the origin because $a,b > 0$ here.  The discussion leading
up to Lemma~\ref{lm:Ga1Ga2int} and Corollary~\ref{cor:Kstar} shows that
for general $k \in \Phi_1$, the value of $\tilde{G}/(a^2b)$ changes sign when
$k$ crosses a point of $K_3$ and at no other time.

The vector with components $(-\frac{a}{b} \sin k_2, \sin k_1)$ points in the
direction of $\xi^\perp$ and does not vanish while $k \in \Phi_1$
(the points where $a=b=0$ are handled by~\eqref{eq:Fder1at0}).
Indeed one could choose this vector as the definition of $\xi^\perp(k)$.
Consider $\xi^\perp$ as measured in polar coordinates. 
The path of $\mathcal{V}(k)$ turns to the left or the right depending on
whether the polar angle of $\xi^\perp$ is increasing or decreasing with $k$.
The direction of change
for this angle in turn depends on the sign of the determinant
\begin{equation*}
\det \begin{pmatrix} 
-\frac{a}{b} \sin k_2 & -\big(\lambda_2 \frac{a(1-b^2)^2}{b^2}
 + \lambda_1\frac{1-a^2}{b}\big)\sin k_1 \\
\sin k_1 & -\lambda_2\frac{a^2(1-b^2)}{b}\sin k_2 \end{pmatrix}
\ = \
\frac{1}{b^2} \big( \lambda_2a(1-b^2)^2 + \lambda_1 b(1-a^2)^2\big).
\end{equation*}
The left column of the $2\times 2$ matrix above is $\xi^\perp(k)$.  The right column is
its rate of change as $k$ travels along $\Phi_1$ at the prescribed velocity, computed
via the product
$\begin{pmatrix} \frac{\sin k_1 \sin k_2}{b} & -\frac{a}{b^2} \\a & 0\end{pmatrix}
\begin{pmatrix} -\partial_b F \sin k_2 \\ \partial_a F \sin k_1 \end{pmatrix}$.

In the analysis of~\eqref{eq:Fder2}, this quantity is shown to have the same
sign as $ab$, which is positive on the loop of $\Phi_1$ corresponding to $\Psi_1$
and negative on the loop corresponding to $\Psi_2$.  

One last note is that the sign of $\frac{a}{b}$ is constant on the connected
components of $\Phi_1$, so the direction of $\xi^\perp$ winds exactly once
around the origin on each loop.  Combined with the preceding facts, it follows
that $\Psi_1$ is a simple closed curve enclosing a convex region in $\R^2$,
and that $\Psi_2$ is a simple closed curve composed of four arcs with the
opposite curvature from $\Psi_1$.  These arcs intersect at cusps located at
the points $V_3$ corresponding to values of $k \in K_3$ where
$\tilde{G}/(a^2b)$ changes sign.

It remains to be shown that $\Psi_1$ is the boundary of 
$\mathcal{V}([-\pi, \pi]^2)$ and that $\Psi_1$, $\Psi_2$ are disjoint.
By comparing supplementary angles, it is clear that the extreme values of
$\mathcal{V}(k)$ in any given direction must occur while $|k_1|, |k_2| \leq
\frac{\pi}{2}$.  With the exception of
$k = (\pm \frac{\pi}{2}, \pm \frac{\pi}{2})$, all points where
$\mathcal{V}(k) \in \Psi_2$ satisfy $\cos k_1 \cos k_2 < 0$, so one 
of the coordinates is necessarily greater than $\frac{\pi}{2}$.  
When $|k_1| = |k_2| = \frac{\pi}{2}$, the vector
$\xi^\perp$ which spans the image of $D\mathcal{V}$ happens to be collinear
with $\mathcal{V}(k)$, so these choices for $k$ do not generate
extreme points of $\mathcal{V}([-\pi,\pi]^2)$ in their respective directions.
By process of elimination, the boundary must consist of $\Psi_1$ alone.

\end{document}